\documentclass{article}

\usepackage{graphicx}
\usepackage{amssymb}
\usepackage{epstopdf}
\usepackage{amsmath}
\usepackage{cite}
\usepackage{amsthm}
\usepackage{mathtools}
\usepackage{enumerate}
\usepackage{booktabs}				

\usepackage{adjustbox}

	\usepackage[margin = 1.0in]{geometry}

\usepackage{animate}

\usepackage[us, 12hr]{datetime}

\DeclareMathOperator{\soft}{soft}
\DeclareMathOperator{\sign}{sign}

\providecommand{\norm}[1]{\lVert#1\rVert}
\providecommand{\abs}[1]{\lvert#1\rvert}

\newcommand{\inv}{^{-1}}
\newcommand{\myJ}{ \mathrm{j} }
\newcommand{\myE}{ \mathrm{e} }
\newcommand{\om}{\omega}
\newcommand{\lam}{{\lambda} }

\newcommand{\half}{\frac{1}{2}}

\newcommand{\iter}[1]{ ^{(#1)} }

\newcommand{\tr}{^{\mathsf{T}}}			
\newcommand{\opt}{^{\mathsf{opt}}}			

\newcommand{\RR}{\mathbb{R}}
\newcommand{\ZZ}{\mathbb{Z}}

\renewcommand{\geq}{\geqslant}
\renewcommand{\le}{\leqslant}
\renewcommand{\ge}{\geqslant}

\newtheorem{defn}{Definition}
\newtheorem{prop}{Proposition}
\newtheorem{lemma}{Lemma}
\newtheorem{theorem}{Theorem}

\newcommand{\mge}{\succcurlyeq}		
\newcommand{\mle}{\preccurlyeq}		

\newcommand{\ellip}{\mathcal{E}}

\newcommand{\maj}{ ^\mathsf{M} }

\newcommand{\SP}{A}		

\newcommand{\SN}{ \Theta }		

\newcommand{\la}{$ \ell_1 $}

\allowdisplaybreaks		

\title{Enhanced Sparsity by Non-Separable Regularization}

	\author{Ivan Selesnick and \.{I}lker Bayram
		\thanks{I. Selesnick, Electrical and Computer Engineering Dept,
	Tandon School of Engineering, New York University, NY, USA.}
		\thanks{\.{I} Bayram, Electronics and Communication Engineering Dept, Istanbul Technical University, Istanbul, Turkey.}
		\thanks{Email: selesi@nyu.edu, ilker.bayram@itu.edu.tr}
		\thanks{This research was supported by NSF under grant CCF-1525398.}%
		\thanks{Last Edit:~\today~\currenttime.}
	}
	\date{}

	\renewcommand\footnotemark{}	

\begin{document}

\maketitle

\begin{abstract}

This paper develops a convex approach for sparse one-dimensional deconvolution
that improves upon L1-norm regularization, the standard convex approach.
We propose a sparsity-inducing non-separable non-convex bivariate penalty function
for this purpose.
It is designed to enable the convex formulation of ill-conditioned linear inverse problems with quadratic data fidelity terms.
The new penalty overcomes limitations of separable regularization. 
We show
how the penalty parameters should be set to ensure that the objective function is convex,
and
provide an explicit condition to verify the optimality of a prospective solution.
We present an algorithm (an instance of forward-backward splitting)
for sparse deconvolution using the new penalty.

\end{abstract}

\section{Introduction}

Methods for sparse regularization
can be broadly categorized as convex or non-convex.
In the standard convex approach, the regularization terms (penalty functions) are convex; and the objective function, 
consisting of both data fidelity and regularization terms, is convex  \cite{Palomar_2010, Bach_2012_NOW}.
The convex approach has several benefits: 
the objective function is free of extraneous local minima,
and
globally convergent optimization algorithms can be leveraged  \cite{Boyd_convex}.

Despite the attractive properties of convex regularization,
non-convex regularization often performs better \cite{Candes_2008_JFAP, Bruckstein_2009_SIAM, Nikolova_2011_chap}.
Classical and recent examples are in edge preserving tomography
\cite{Geman_1992_PAMI, Charbonnier_1997_TIP, Nikolova_1999_TIP, Nikolova_2010_TIP}
and compressed sensing \cite{Trzasko_2009_TMI, Chartrand_2009_ISBI, Chartrand_2013_Asilomar},
respectively.
In the non-convex approach, penalty functions are non-convex
as they can be designed to induce sparsity more effectively than convex ones.
However, the convexity of the objective function is generally sacrificed.
Consequently, non-convex regularization is hampered by complications:
the objective function
will generally possess many sub-optimal local minima in which optimization algorithms can become entrapped. 

It turns out, non-convex penalties can be utilized
without giving up the convexity of the objective function and corresponding benefits.
This is achieved by carefully specifying the penalty in accordance with the data fidelity term,
as described by Blake, Zimmerman, and Nikolova
 \cite{Blake_1987, Nikolova_1998_ICIP, Nikolova_1999_TIP, Nikolova_2010_TIP}.
In recent work, a class of sparsity-inducing non-convex penalties has been developed 
to formulate convex objective functions
and applied to several signal estimation problems \cite{Parekh_2015_SPL, Ding_SPL_2015, Selesnick_SPL_2015, Selesnick_2014_TSP_MSC, Chen_2014_TSP_ncogs, Bayram_2014_ICASSP,Parekh_2015_arXiv_FLSA, Lanza_2015_LNCS, Bayram_2015_SPL}.
This approach maintains the benefits of the convex framework (absence of spurious local minima, etc.),
yet estimates sparse signals more accurately than convex regularization (e.g., the \la\ norm)
due to the sparsity-inducing properties of non-convex regularization.
However, this previous work considers only \emph{separable} (additive) penalties, 
which have fundamental limitations.

In this paper, we introduce a parameterized sparsity-inducing non-separable non-convex bivariate penalty function.
The penalty is designed to enable the convex formulation of ill-conditioned linear inverse problems with quadratic data fidelity terms.
The new penalty overcomes limitations of separable non-convex regularization. 
We show
how the penalty parameters should be set to ensure the objective function is convex.
We also show
how this bivariate penalty can be incorporated into linear inverse problems of $ N $ variables ($ N > 2 $),
and we provide an explicit condition to verify the optimality of a prospective solution.
We derive two iterative algorithms for optimization using the new penalty, 
and demonstrate its effectiveness for one-dimensional sparse deconvolution.

\subsection{Basic problem statement}

We consider the problem of bivariate sparse regularization (BISR) 
with a quadratic data fidelity term:
\begin{equation}
	\label{eq:biprob}
	\hat x = 
	\arg \min_{x \in \RR^2} \bigg\{ f(x) = \half \norm{ y - H x }_2^2 + \lam \psi(x) \bigg\}	
\end{equation}
where $ \lam > 0 $,
$ H $ is a $ 2 \times 2 $ matrix,
and $ \psi \colon \RR^2 \to \RR $ is a bivariate penalty. 
[In Sec.~\ref{sec:recon}, it will be shown how to 
extend this bivariate problem to an $ N $-point linear inverse problem.] 
In this paper, we suppose $ H\tr H $ is Toeplitz,
as this naturally arises in deconvolution problems. 
Correspondingly, 
we write $ H\tr H = K(\gamma) $ where $ \gamma = (\gamma_1, \gamma_2) \in \RR^2 $ and
\begin{equation}
	\label{eq:defK}
	K(\gamma)
	:=
	\half
	\begin{bmatrix}
		\gamma_1 + \gamma_2 &  \gamma_1 - \gamma_2
		\\
		\gamma_1 - \gamma_2 & \gamma_1 + \gamma_2
	\end{bmatrix}
	= Q \, \Gamma \, Q\tr
\end{equation}
where
\begin{equation}
	\label{eq:defQG}
	Q =
	\frac{1}{\sqrt{2}}
	\begin{bmatrix}
		1 & 1
		\\
		1 & -1 
	\end{bmatrix},
	\quad
	\Gamma = 
	\begin{bmatrix}
		\gamma_1 & 0
		\\
		0 & \gamma_2 
	\end{bmatrix}.
\end{equation}
This is an eigenvalue decomposition of $ H\tr H $.
The parameters $ \gamma_1 $ and $ \gamma_2 $ are the eigenvalues 
of the positive semidefinite matrix $ H\tr H $;
hence, they are nonnegative. 

First, suppose $ \psi $ is 
a separable convex penalty, e.g., $ \psi(x) = \abs{x_1} + \abs{x_2} $
corresponding to the \la\ norm (Fig.~\ref{fig:pensep}(a)).
Then the objective function $ f $ in \eqref{eq:biprob} is convex,
but it does not induce sparsity as effectively as non-convex penalties can. 
In particular, $ \ell_1 $ norm regularization tends to underestimate the true signal values.

Second, suppose $ \psi $ is a separable non-convex penalty,
i.e., $ \psi(x) = \phi(x_1) + \phi(x_2) $, as illustrated in Fig.~\ref{fig:pensep}(b).
Then the objective function $ f $ is convex only if $ \phi $ is suitably chosen.
In particular, 
for the class of penalties we consider, 
the objective function is convex only if $ \phi''(t) \ge -\min\{\gamma_1, \gamma_2\}/\lam $,
where $ \gamma_1 $ and $ \gamma_2 $ are the eigenvalues of $ H\tr H $.
(See Lemma \ref{lemma:pensep}, \cite{Selesnick_2014_TSP_MSC}, and \cite{Blake_1987}.)
When $ H $ is singular, the minimum eigenvalue of $ H\tr H $ is zero
and $ \phi $ must be convex
(i.e., it induces sparsity relatively weakly).
Hence, if $ H $ is singular and we restrict the penalty $ \psi $ to be separable,
then we can not use sparsity-inducing non-convex regularization without 
sacrificing the convexity of the objective function $ f $.
Indeed, when separable non-convex penalties are utilized for their strong sparsity-inducing properties,
the convexity of $ f $ is generally sacrificed.

Our aim is to prescribe a non-separable penalty $ \psi $ so that 
the objective function is guaranteed to be convex even though the penalty $ \psi $ itself is not.
Such a penalty will be given in Sec.~\ref{sec:pen}. 
It turns out, when we utilize a \emph{non-separable} non-convex penalty to strongly induce sparsity, 
we need \emph{not} sacrifice the convexity of the objective function $ f $,
even when $ H $ is \emph{singular}.

\begin{figure}[t]
	\centering
		\begin{tabular}{c@{\hspace{4em}}c}
		\includegraphics{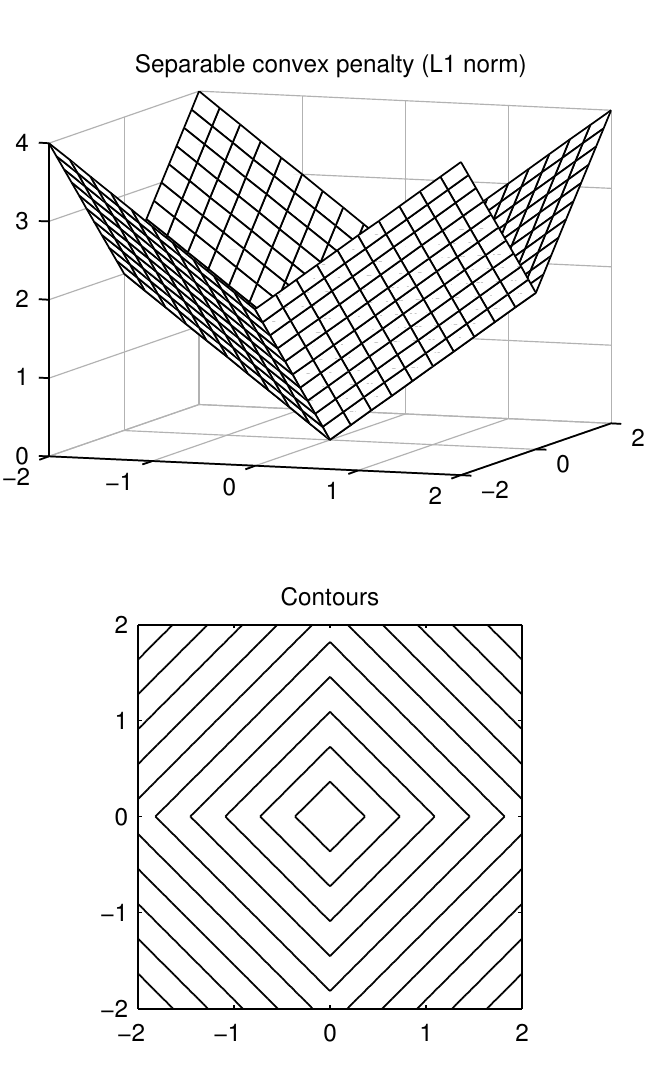}
		&
		\includegraphics{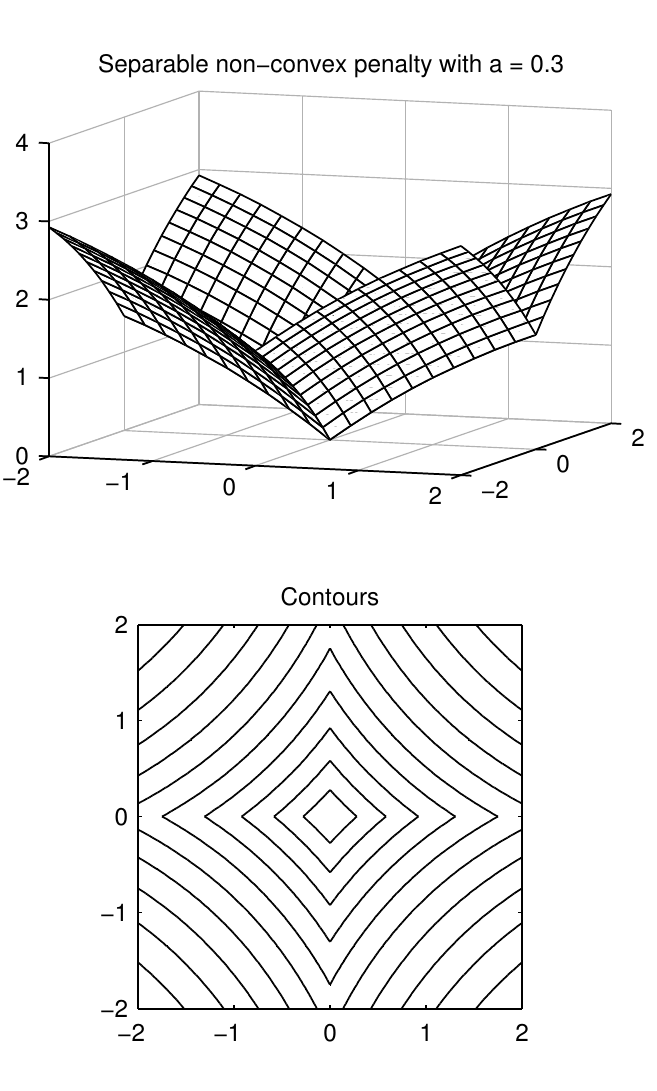}
		\\
		(a)
		&
		(b)
		\end{tabular}
	\caption{
		Separable bivariate penalties, convex and non-convex.
	}
	\label{fig:pensep}
\end{figure}

\subsection{Related work}

The design of non-convex regularizers ensuring convexity of an objective function
was
proposed as part of the Graduated Non-Convexity (GNC) approach \cite{Blake_1987, Nikolova_1999_TIP, Nikolova_2010_TIP}
and for binary image estimation \cite{Nikolova_1998_ICIP}.
Most methods for non-convex sparse regularization do not aim to maintain convexity of the objective function.
The $ \ell_p $ pseudo-norm ($ 0 \le p < 1 $) is widely used,
but other regularizers have also been advocated
\cite{Zou_2008_AS, Lorenz_2007, Gasso_2009_TSP, Gholami_2011_TSP, Voronin_2013_ICASSP, Malioutov_2014_icassp,  Geman_1992_PAMI, Mohimani_2009_TSP, Marnissi_2013_ICIP, Chartrand_2009_ISBI, Chartrand_2014_ICASSP}.
Several algorithms have been developed specifically for the $ \ell_0 $ pseudo-norm:
matching pursuit \cite{Mallat_1998},
greedy $\ell_1$ \cite{Kozlov_2010_geomath},
iterative hard thresholding and its variations \cite{Blumensath_2012_SP, Kingsbury_2003_ICIP,Portilla_2007_SPIE, Qiu_2012_TSP_ECME, Foucart_2010_SIAM, Wang_2010_SIAM},
smoothed $\ell_0$, \cite{Mohimani_2009_TSP},
and single best replacement \cite{Soussen_2011_TSP}.

The penalty developed here is different from regularizers described in the literature, 
most of which are separable or are compositions of separable regularizers and linear operators.
Non-separable penalties
(e.g., mixed norms \cite{Marnissi_2013_ICIP, SenSel-02-tsp, Bach_2012_NOW})
are generally used to capture statistical relationships among signal values
(possibly linearly transformed%
\footnote{Isotropic two-dimensional total variation,
which induces joint sparsity of horizontal and vertical gradients, 
exemplifies this type of regularization.})
or to induce structured sparsity;
they are not intended for 
unstructured, uncorrelated, `pure'  sparsity.
Moreover, most sparsity-inducing penalties 
are neither
parameterized nor utilized in the way undertaken here: namely, to induce sparsity subject to the constraint that the objective function is convex. 

As the proposed method leads to a convex problem, 
convex sparse optimization methods may be used or adapted for its solution.
Representative algorithms are
the iterative shrinkage/thresholding algorithm (ISTA/FISTA)
\cite{Fig_2003_TIP,Beck_2009_SIAM}, 
proximal methods
\cite{Combettes_2008_SIAM, Combettes_2011_chap},
alternating direction method of multipliers (ADMM)
\cite{Afonso_2010_TIP_SALSA, Goldstein_2009_SIAM},
and 
majorization-minimization (MM) \cite{FBDN_2007_TIP, Lange_2014_ISR}.

Several algorithms are suitable for general non-convex sparse regularization problems, 
such as
iteratively reweighted least squares (IRLS) \cite{Harikumar_1996_cnf, Daubechies_2010_CPAM},
iteratively reweighted \la\ (IRL1) 
\cite{Wipf_2010_TSP, Asif_2013_TSP, Candes_2008_JFAP},
FOCUSS
\cite{Rao_2003_TSP},
related algorithms
\cite{Mourad_2010_TSP, Tan_2011_TSP, Montefusco_2013_SP},
non-convex MM
\cite{Gholami_2011_TSP, Voronin_2013_ICASSP, Malioutov_2014_icassp, Chouzenoux_2013_SIAM},
extensions 
of GNC
\cite{Nikolova_1999_TIP, Nikolova_2010_TIP, Mohimani_2009_TSP},
and other methods
\cite{Charbonnier_1997_TIP, Gasso_2009_TSP, Chartrand_2009_ISBI}.
Non-convex regularization has also been used for blind deconvolution \cite{Repetti_2015_SPL}
and low-rank plus sparse matrix decompositions \cite{Chartrand_2012_TSP}.
Convergence of these algorithms are generally to local optima only.
However, 
conditions for convergence to a global minimizer, 
or to guarantee that all local minimizers are near a global minimizer, 
have been recently reported \cite{Chen_2014_TSP_Convergence, Loh_2015_JMLR}.
Whereas Refs.~\cite{Chen_2014_TSP_Convergence, Loh_2015_JMLR} focuses on convergence guarantees for given penalties, 
we focus on the design of penalties.

\subsection{Notation}
\label{sec:notation}

We write the vector $ x \in \RR^N $ as $ x = (x_1, \ x_2, \ \dots, \ x_{N}) $.
Given $ x \in \RR^N $, we define $ x_n = 0 $ for $ n \notin \{ 1 , 2 , \dots, N \} $.
(This simplifies expressions involving summations over $ n $.)
The \la\ norm of $ x \in \RR^N $ is defined as $ \norm{ x }_1 = \sum_n \abs{ x_n } $.
If the matrix $ A $ is positive semidefinite, we write $ A \mge 0 $. 
If the matrix $ A - B $ is positive semidefinite, we write $ A \mge B $. 

\begin{table}
	\caption{\label{table:penalties}%
	Penalties}
	\begin{center}
	\begin{tabular}{@{} l @{}} 
		\toprule 
		$
		\displaystyle
		\phi(t; a) =
		\frac { \abs{ t } } { 1 + a \abs{ t } / 2 } , 
		\quad a  \ge 0
		$
		\\[1.5em]
		$
	\phi(t; a) =
	\begin{cases}
		 \frac{1}{a} \, \log(1 + a \abs{t}), \ & a > 0
		\\
		\abs{t},	& a = 0
	\end{cases}
		$
		\\[1.5em]
		$
	\phi(t; a) =
	\begin{cases}
		\frac{2}{a\sqrt{3}}
			\Bigl(
				\tan\inv \Bigl( \frac{1+2 a \abs{t}}{\sqrt{3}} \Bigr) - \frac{\pi}{6}
			\Bigr),
		\ & a > 0
		\\
		\abs{t},	& a = 0
	\end{cases}
		$
		\\
	\bottomrule 
	\end{tabular} 
\end{center}
\end{table}

\section{Univariate Penalties}
\label{sec:pen1}

The bivariate penalty to be given in Sec.~\ref{sec:pen} will be 
based on a parameterized non-convex univariate penalty function
$ \phi(\, \cdot \, ; a) \colon \RR \to \RR $ 
with parameter $ a \ge 0 $.
We shall assume $ \phi $ has the following properties:
\begin{enumerate}[ P1)]
\setlength{\itemsep}{0cm}%
\setlength{\parskip}{0cm}%
  \item
$ \phi( \, \cdot \,; a) $ is continuous on $ \RR $
\item
$ \phi(\, \cdot \,; a) $ is twice continuously differentiable, increasing, and concave on $ \RR_+ $ 
\item
$ 	\phi(0; a)  = 0 $ 
\item
$ \phi(t; 0) = \abs{t} $
\item
$ 	\phi(-t; a)  = \phi(t; a) $ 
\item
$ \phi'(0^+; a) = 1 $ 
\item
$ \phi''(0^+; a) = -a  $
\item
$ \phi''(t; a) \geq -a  $ for all $ t \neq 0 $
\item
$ \phi(t; a) $ is decreasing and convex in $ a $.
\item
$ \phi(t; a) = (b/a) \, \phi\big(at / b; b\big) $ for $ a, b > 0 $  
[Scaling].
\end{enumerate}

It follows from symmetry that $ \phi'(-t) = -\phi'(t) $ and $ \phi''(-t) = \phi''(t) $.
The scaling property of $ \phi $ also induces a scaling property of $ \phi' $ and $ \phi'' $. 
Namely,
\begin{align}
	\label{eq:phidscale}
	\phi'(t; a) & = \phi'(a t / b;  b), 
	\\
	\label{eq:phiddscale}
	\phi''( t ; a)
	& = 
	(a/b) \, \phi''( a t / b ; b).
\end{align}

Table \ref{table:penalties} lists several penalty functions that satisfy the above properties.
For example,
the rational \cite{Geman_1992_PAMI},
logarithmic, and arctangent functions
\cite{Candes_2008_JFAP, Nikolova_2011_chap, Selesnick_2014_TSP_MSC}
(when suitably normalized).
The arctangent penalty is illustrated in Fig.~\ref{fig:penalties_1D_multi} for several values of $ a $.
For larger $ a $, the penalty functions increase more slowly, are more concave on the positive real line,
and induce sparsity more strongly (i.e., by mildly penalizing large values). 
A comparison of the three penalties listed in Table \ref{table:penalties} is illustrated in 
Figure~1 of Ref.~\cite{Chen_2014_TSP_ncogs} for a fixed value of $ a $.
Of the three penalties, 
the arctangent penalty increases the slowest for a fixed value of $ a $.

We mention that we do not use the simpler form of the arctangent penalty, 
$	\phi( t ; a ) = ( 1 / a ) \tan\inv( a \abs{t} ) $,
as it does not satisfy $ \phi''(0^+, a) = -a $ which is property P7 listed above. 

\begin{figure}[t]
	\centering
	\includegraphics{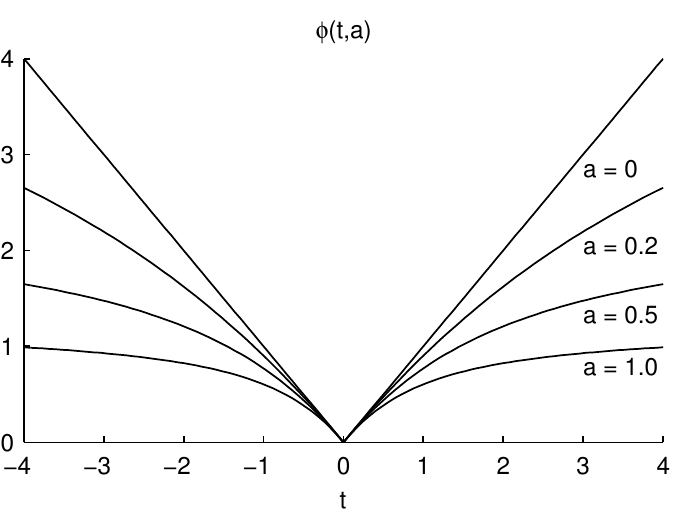}
	\caption{
		Penalty $ \phi( t ; a) $ for several values of $ a $.
	}
	\label{fig:penalties_1D_multi}
\end{figure}

\begin{figure}
	\centering
		\includegraphics{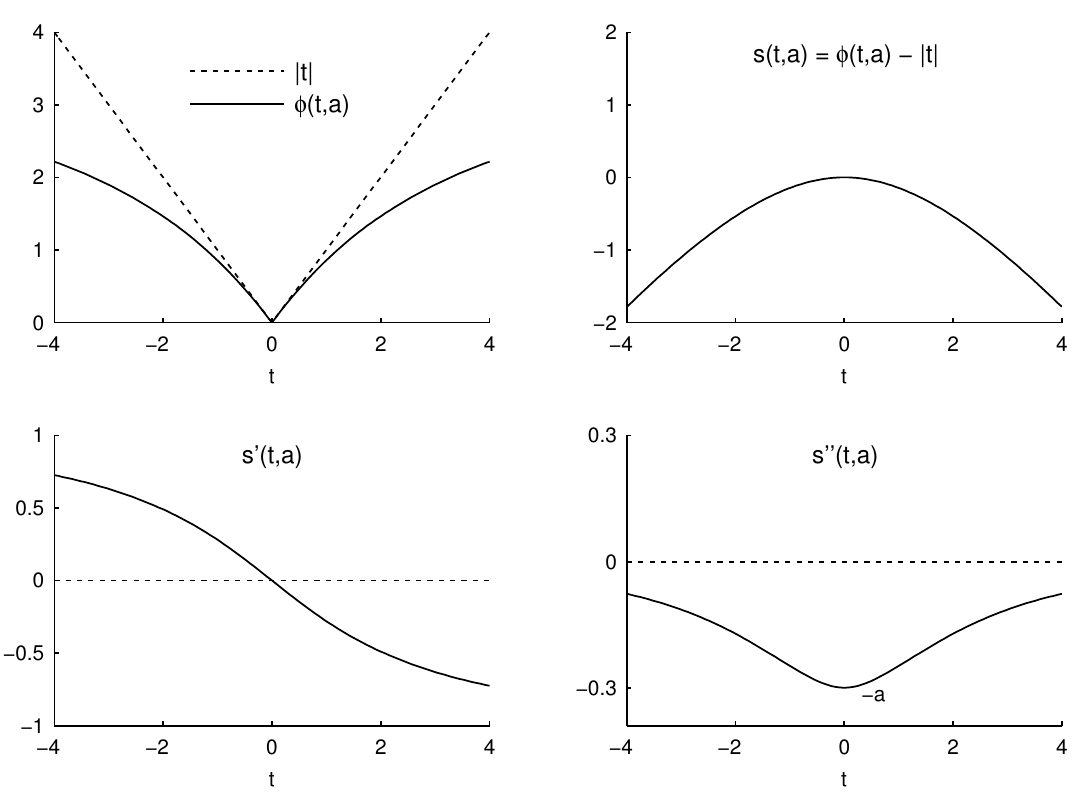}
	\caption{
		A univariate penalty $ \phi $, its corresponding function $ s $, 
		and the first and second-order derivatives of $ s $.
		The function $ s $ is twice continuously differentiable and concave.
	}
	\label{fig:penalties_1D}
\end{figure}

Corresponding to a penalty $ \phi $ having the above properties, we define a smooth concave function.

\begin{defn}
\label{def:s}
Let $ \phi \colon\RR \to \RR $ be a penalty function satisfying the properties listed above. 
For $ a \ge 0 $, 
we define $ s \colon \RR \to \RR $,
\begin{equation}
	\label{eq:defs}
	s(  t  ; a ) = \phi(  t  ; a ) - \abs{  t  }.
\end{equation}
\end{defn}

Figure~\ref{fig:penalties_1D} illustrates the function $ s $ corresponding to the arctangent penalty for $ a = 0.3 $. 
The following proposition follows straightforwardly \cite{Parekh_2015_SPL}.

\begin{prop}
\label{prop:sfun1d}
Let $ a \ge 0 $.
Let $ \phi \colon\RR \to \RR $ be a penalty function satisfying the properties listed above. 
The function $ s \colon\RR \to \RR $ in Definition \ref{def:s}
is twice continuously differentiable, concave, and satisfies
\begin{equation}
	\label{eq:sbound}
	-a \le s''( t ; a) \le 0.
\end{equation}
\end{prop}
This property will be of particular importance. 
Note that the value $ \phi''(0) $ is not defined since $ \phi $ is not differentiable at zero.
But the value $ s''(0) $ is defined (and is equal to $ -a $).
\begin{equation}
	\label{eq:s2a}
	s''( t ; a ) =
	\begin{cases}
		-a, 	\ \ & t = 0
		\\
		\phi''( t ; a ) & t \neq 0.
	\end{cases}
\end{equation}
Also, although $ \phi'(0) $ is not defined [because $ \phi'(0^+) = 1 $ and $ \phi'(0^-) = -1 $],
the value $ s'(0) $ is defined [$ s'(0) = 0 $].
Many of the properties listed above for $ \phi $ are inherited by $ s $,
such as the symmetry and scaling properties:
\begin{align}
	s(-t ; a) & = s(t ; a)
	\\
	\label{eq:scaling_s}
	s(t; a) & = (b/a) \, s\big(at / b; b\big)
	\\
	\label{eq:ssp1}
	s'(t; a) & = s'\big(at / b; b\big)
	\\
	\label{eq:scaling_sd2}
	s''(t; a) & = (a/b) \, s''\big(at / b; b\big)
\end{align}
The following proposition is proven in Appendix \ref{app:proofpen1}.
\begin{prop}
\label{prop:gfconvex}
Let $ \phi \colon\RR \to \RR $ satisfy the properties listed above. 
Let $ s \colon\RR \to \RR $ be given by Definition \ref{def:s}.
Let $ \lam > 0 $.
If $ 0 \le a \le 1/\lam $,
then the functions
$ g \colon\RR \to \RR $
\begin{equation}
	\label{eq:defgprop}
	g( t ) = 
		\half  t ^2 + \lam s( t ; a)
\end{equation}
and
$ f \colon\RR \to \RR $
\begin{equation}
	f( t ) = 
		\half  t ^2 + \lam \phi( t ; a)
\end{equation}
are convex functions.
\end{prop}

\begin{figure}
	\centering
		\begin{tabular}{c@{\hspace{2em}}c}
		\includegraphics{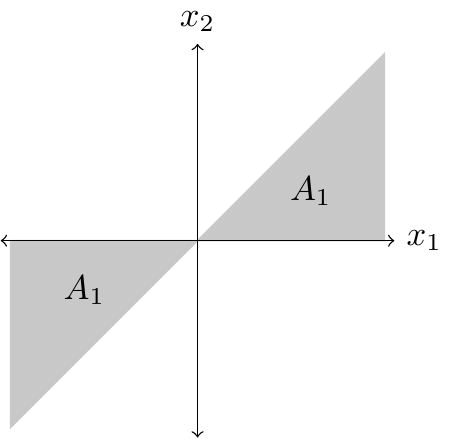}
		&
		\includegraphics{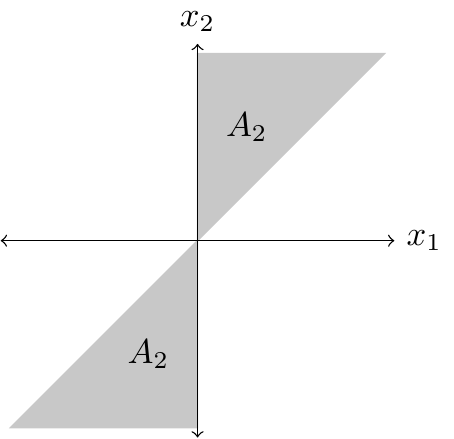}
		\\[2em]
		\includegraphics{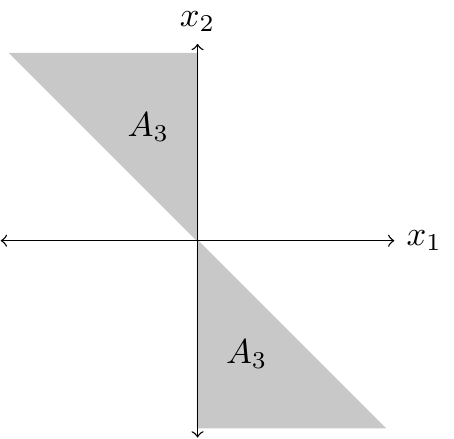}
		&
		\includegraphics{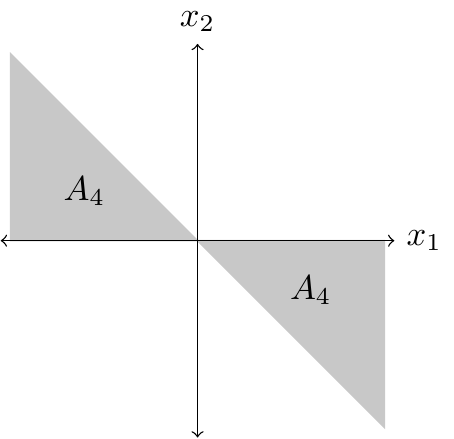}
		\end{tabular}
	\caption{Regions $ \SP_1 $ through $ \SP_4 $ in Definition \ref{def:Sfun}.}
	\label{fig:defA}
\end{figure}

\section{Bivariate Concave Function}
\label{sec:concave}

The bivariate penalty to be given in Sec.~\ref{sec:pen} will be defined in terms of a concave bivariate function $ S $.
The role of $ S $, in describing the bivariate penalty, will be analogous to
the role of $ s $ in describing the univariate penalty $ \phi $.
Accordingly, the properties of $ S $ will be important for the properties of the bivariate penalty. 

\begin{defn}
\label{def:Sfun}
Let $ a = (a_1, a_2) $ with $ a_i \ge 0 $.
Let $ \phi \colon\RR \to \RR $ be a univariate penalty function having the properties listed in Sec.~\ref{sec:pen1}.
Let $ s \colon\RR \to \RR $ be given by Definition \ref{def:s}.
If at least one of $ \{ a_1, a_2 \} $ is non-zero, 
we define the function 
 $ S \colon \RR^2 \to \RR $ as
\[
	S( x ; a ) = 
	\begin{cases}
		s( x_1 + r x_2 ; \alpha ) + (1 - r) \, s( x_2 ; a_1 ) ,  & x \in \SP_1
		\\
		s( r x_1 + x_2 ; \alpha ) + (1 - r) \, s( x_1 ; a_1 ) ,  & x \in \SP_2
		\\
		s( r x_1 + x_2 ; \alpha ) + (1 + r) \, s( x_1 ; a_2 ) ,  & x \in \SP_3
		\\
		s( x_1 + r x_2 ; \alpha ) + (1 + r) \, s( x_2 ; a_2 ) ,  & x \in \SP_4
	\end{cases}
\]
where
\begin{equation}
	\label{eq:ralpha}
	\alpha = \frac{ a_1 + a_2 }{2},
	\qquad
	r = \frac{ a_1 - a_2 }{ a_1 + a_2 }
\end{equation}
and sets $ \SP_i \subset \RR^2 $ are
\begin{align}
	\SP_1 & = \{ x \in \RR^2 \mid x_2 ( x_1 - x_2 ) \ge 0 \}
	\\
	\SP_2 & = \{ x \in \RR^2 \mid x_1 (x_1 - x_2) \le 0 \}
	\\
	\SP_3 & = \{ x \in \RR^2 \mid x_1 (x_1 + x_2) \le 0 \}
	\\
	\SP_4 & = \{ x \in \RR^2 \mid x_2 (x_1 + x_2) \le 0 \}
\end{align}
as shown in Fig.~\ref{fig:defA}.
If both $ a_i = 0 $, we define $ S(x ; 0 ) = 0 $.
\end{defn}

The non-negative parameters $ a_i $ characterize
how strongly concave $ S $ is.
Figure~\ref{fig:Sfun}
illustrates $ S $ for the parameter values $ a_1 = 1.5 $ and $ a_2 = 0.3 $.
The level sets of $ S $ are not ellipses, even though they appear ellipsoidal.

The function $ S $ has three symmetries:
\begin{subequations}
\label{eq:syms}
\begin{align}
	\label{eq:sym1}
	S( x_1, x_2 ; a )
	& = 
	S( x_2, x_1 ; a )
	\\
	\label{eq:sym2}
	& =
	S( -x_1, -x_2 ; a )
	\\
	\label{eq:sym3}
	& =
	S( -x_2, -x_1 ; a )
\end{align}
\end{subequations}
(any two of which imply the remaining one).
Equivalently, 
$ S $ is symmetric with respect to the origin and the two lines $ x_1 = x_2 $ and $  x_1 = -x_2 $.
These symmetries follow directly from Definition \ref{def:Sfun} 
and from the symmetry of the univariate function $ s $.

\begin{figure}[t]
	\centering
		\includegraphics{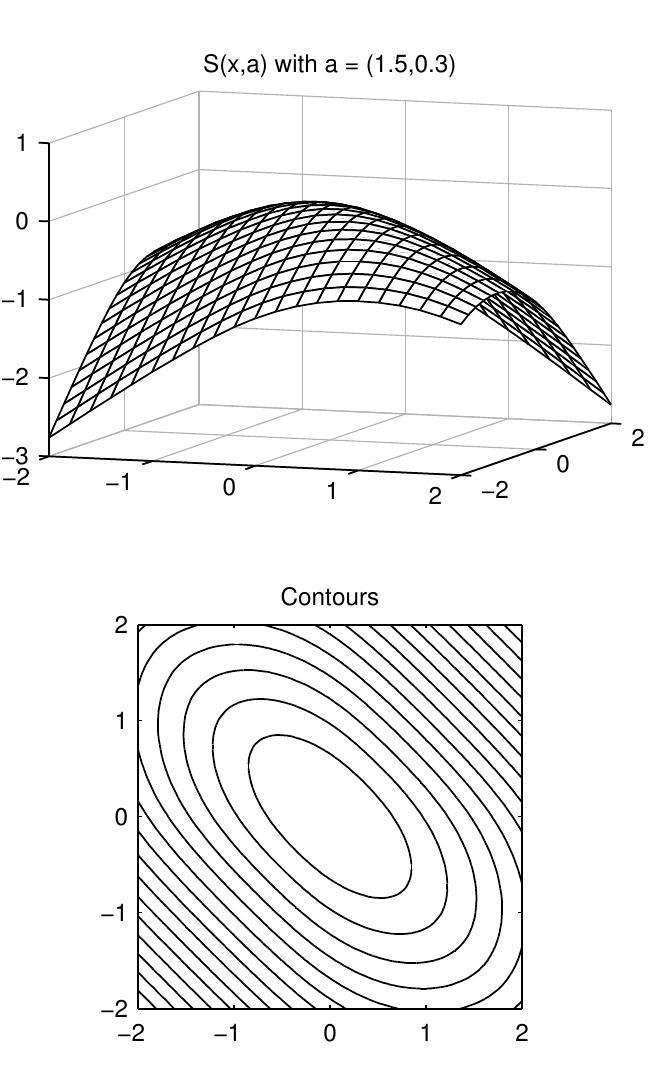}
	\caption{
		Non-separable concave function $ S $ in Definition \ref{def:Sfun}. }  
	\label{fig:Sfun}
\end{figure}

The following lemmas are proven in the Appendix. 
It will be useful in the proofs to note some identities.
First, note that  $ S(0; a) = 0 $.
From the definitions of $ \alpha $ and $ r $ in  \eqref{eq:ralpha}, we have:
\begin{alignat}{2}
	\abs{ r } & \le 1,
	&
	\label{eq:a1id}
	(1 + r) \, \alpha & = a_1
	\\
	\label{eq:adiffid}
	r \alpha & = (a_1 - a_2) / 2,
	&
	\qquad
	(1 - r) \, \alpha & = a_2.
\end{alignat}

\begin{lemma}
\label{lemma:defS}
The bivariate function $ S \colon \RR^2 \to \RR $ 
in Definition \ref{def:Sfun}
is twice continuously differentiable and concave on $ \RR^2 $.
\end{lemma}

\begin{lemma}
\label{lemma:Shessian}
Let $ a = (a_1, a_2) $ with $ a_i \ge 0 $.
The Hessian of the bivariate function $ S $ in Definition \ref{def:Sfun}
satisfies
\begin{equation}
	\label{eq:ellipse}
	- \half
	\begin{bmatrix}
		a_1 + a_2 &  a_1 - a_2
		\\
		a_1 - a_2 & a_1 + a_2
	\end{bmatrix}
	\mle
	\nabla^2 S(x ; a )
	\mle 
	0,
	\quad
	\text{for all} \ 
	x \in \RR^2.
\end{equation}
Equivalently,
\begin{equation}
	\label{eq:ellipseB}
	-K(a)
	\mle 
	\nabla^2 S(x ; a )
	\mle 
	0,
	\quad
	\text{for all} \ 
	x \in \RR^2
\end{equation}
where $ K(a) $ is defined by its eigenvalue decomposition
\begin{equation}
	K(a)
	:= 
	Q\tr
	\begin{bmatrix}
		a_1 & 0
		\\
		0 & a_2
	\end{bmatrix}
	Q	
	=
	\half
	\begin{bmatrix}
		a_1 + a_2 &  a_1 - a_2
		\\
		a_1 - a_2 & a_1 + a_2
	\end{bmatrix}
\end{equation}
where $ Q $ is the orthonormal matrix
defined in \eqref{eq:defQG}.
Furthermore, the lower bound is attained at $ x = 0 $, i.e.,
\begin{equation}
	\label{eq:SKZ}
	\nabla^2 S( 0 ; a ) = -K(a).
\end{equation}
\end{lemma}

Lemma \ref{lemma:Shessian} states that 
$ S $ is maximally concave at the origin.
The lemma also gives the Hessian at the origin in terms of the parameters $ a_i $.
Lemma \ref{lemma:Shessian} is a key result for the subsequent results. 

Lemma \ref{lemma:Shessian} can be illustrated in terms of ellipses.
If $ M $ is a positive semidefinite matrix, then the set $ \ellip[M] = \{ x  : x\tr M\inv x \le 1 \} $ is an ellipsoid \cite{Boyd_convex}.
In addition, $ M_1 \mle M_2 $ if and only if $ \ellip[M_1] \subseteq \ellip[M_2] $.
In Fig.~\ref{fig:ellipses}, 
we set $ a_1 = 1.5 $ and $ a_2 = 0.3 $.
The ellipses corresponding to $ K(a) $ and $ -\nabla^2 S( x ; a ) $
are shown in gray and black, respectively. 
For each $ x $, the black ellipse is contained within the gray ellipse, illustrating \eqref{eq:ellipseB}.
For large $ x $,
the black ellipse shrinks,
indicating that the function $ S $ becomes less concave away from the origin. 
The shrinkage behavior of the ellipse for large $ x_i $ is different in different quadrants:
it shrinks faster in quadrants 1 and 3 than in quadrants 2 and 4.
At $ x = 0 $, the black and gray ellipses coincide,
reflecting the fact that the Hessian of $ S $ is equal to $ -K(a) $ at the origin \eqref{eq:SKZ}.

\begin{figure}[t]
	\centering
		\includegraphics{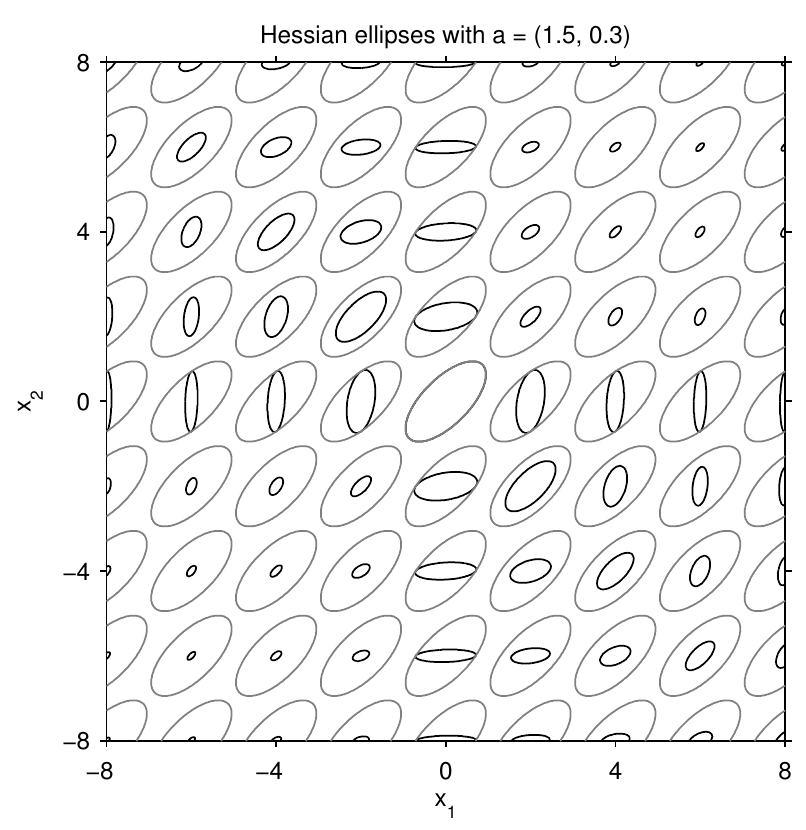}
	\caption{
		Illustration of Lemma \ref{lemma:Shessian}.
		The black ellipses are inside the gray ellipses,
		indicating that
		$ 0.5 x\tr K x + S(x; a )$ is convex.
	}
	\label{fig:ellipses}
\end{figure}

\begin{theorem}
\label{thm:cvx2}
Let $ a = (a_1, a_2) $ with $ a_i \ge 0 $.
Let $ S \colon \RR^2 \to \RR $ be the function
in Definition \ref{def:Sfun}.
Let
$	K(\gamma) = Q \, \Gamma \, Q\tr $
with eigenvalues $ \gamma_i \ge 0 $
be the positive semidefinite matrix
defined in \eqref{eq:defK}.
The function $ g \colon \RR^2 \to \RR $,
\begin{equation}
	\label{eq:defg}
	g(x; a, \gamma) = \half x\tr K(\gamma) \, x + \lam \, S(x; a),
	\ 
	\lam > 0,
\end{equation}
is convex
if
\begin{equation}
	\label{eq:aicond}
	0 \le a_1 \le \gamma_1 / \lam,
	\qquad
	0 \le a_2 \le \gamma_2 / \lam.
\end{equation}
\end{theorem}

\begin{proof}
Since $ g $ is twice continuously differentiable, it is sufficient to show
the Hessian of $ g $ is positive semidefinite. 
The Hessian of $ g $ is given by
\begin{align}
	\label{eq:deffH}
	\nabla^2 g(x; a , \gamma )
	& = K(\gamma) + \lam [ \nabla^2 S(x; a) ]
	\\
	& = Q\tr \Gamma \, Q + \lam [ \nabla^2 S(x; a) ].
\end{align}
From Lemma \ref{lemma:Shessian} it follows that
\begin{equation}
	\nabla^2 g(x; a, \gamma )
	\mge
	Q\tr
		\begin{bmatrix}
		\gamma_1 - \lam a_1 & 0
		\\
		0 & \gamma_2 - \lam a_2
	\end{bmatrix}
	Q.
\end{equation}
Hence, 
$ g(x; a, \gamma ) $ is convex if $ \gamma_i - \lam a_i \ge 0 $ for $ i = 1, 2 $.
This proves the result.
\end{proof}

Theorem \ref{thm:cvx2} is illustrated in Fig.~\ref{fig:convexity}.
In this example, we set
$ \gamma_1 = 1.5 $, $ \gamma_2 = 0.3 $, and $ \lambda = 15.0 $.
Hence, the critical parameters are $ a_1^{*} = \gamma_1 / \lam = 0.1 $ and $ a_2^{*} = \gamma_2 / \lam = 0.02 $.
If either $ a_1 $ or $ a_2 $ is greater than the 
respective critical value, then the function $ g $ in \eqref{eq:defg} will be non-convex. 
In Fig.~\ref{fig:convexity}(a),
we set
$ a_i = 0.9 \gamma_i / \lam < a_i^{*} $
to satisfy condition \eqref{eq:aicond},
hence  $ g $ is convex.
In contrast, 
in Fig.~\ref{fig:convexity}(b),
we set
$ a_i = 1.1 \gamma_i / \lam > a_i^{*} $
violating condition \eqref{eq:aicond},
hence 
$ g $ is non-convex.
The lack of convexity can be recognized in both the surface and contour plots.

\begin{figure}[t]
	\centering
	\begin{tabular}{c@{\hspace{4em}}c}
	\includegraphics[]{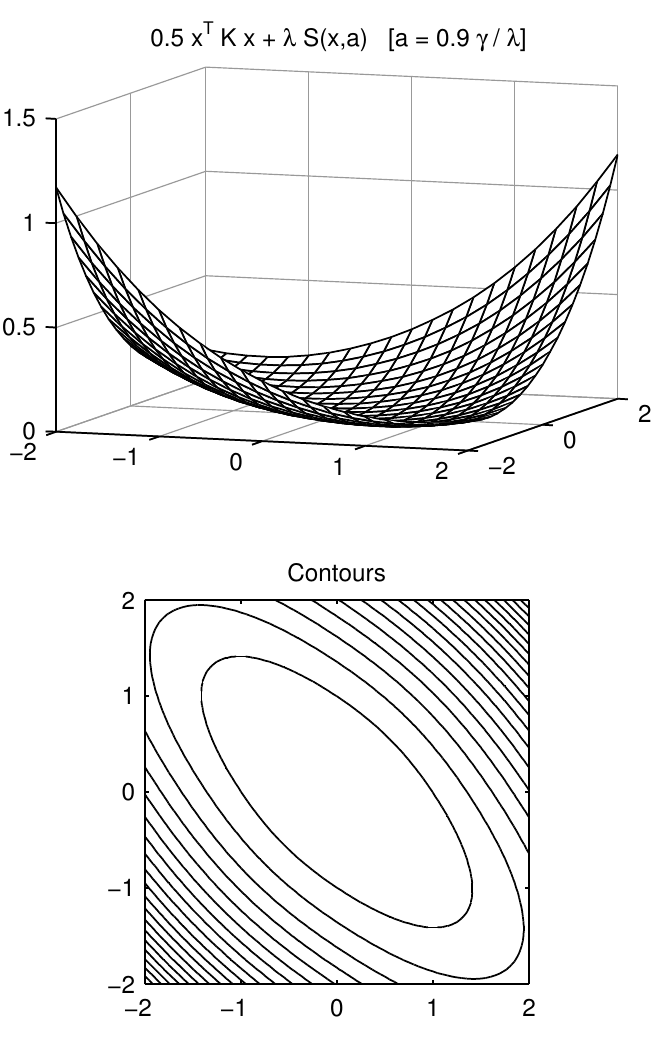}
	&
	\includegraphics[]{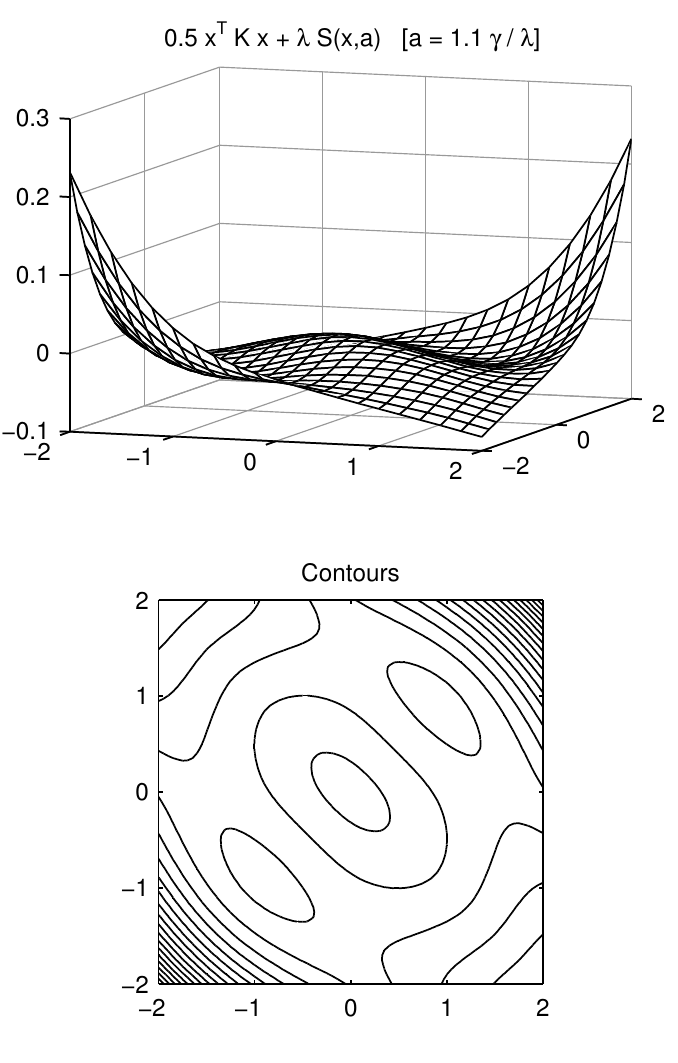}
	\\
	(a)
	&
	(b)
	\end{tabular}
	\caption{
		Illustration of convexity condition \eqref{eq:aicond}.
		(a)
		Function $ g $ is convex as $ a_i $ satisfy \eqref{eq:aicond}.
		(b)
		Function $ g $ is not convex as $ a_i $ violate \eqref{eq:aicond}.
	}
	\label{fig:convexity}
\end{figure}

\section{Bivariate Penalties}
\label{sec:pen}

In this section, we define a non-convex non-separable bivariate penalty. 
Our intention is to strongly induce sparsity in 
solutions of problem \eqref{eq:biprob} while maintaining the convexity of the problem.
The penalty is parameterized by two non-negative parameters $ a_1 $ and $ a_2 $,
which we restrict so as to ensure convexity of the objective function. 

\begin{defn}
\label{def:bp}
Let $ a = (a_1, a_2) $ with $ a_i \ge 0 $.
Let $ \phi \colon\RR \to \RR $ be a univariate penalty function having the properties listed in Sec.~\ref{sec:pen1}.
Let $ S \colon\RR^2 \to \RR $ be 
the corresponding function in Definition \ref{def:Sfun}.
We define the bivariate penalty function 
 $ \psi \colon \RR^2 \to \RR $ as
\begin{equation}
	\label{eq:defpsi}
	\psi( x ; a ) 
	= S( x ; a ) + \norm{ x }_1.
\end{equation}
\end{defn}

If $ a_1 \neq a_2 $, then the penalty $ \psi $ is non-separable.
Figure~\ref{fig:pen2}
illustrates $ \psi $ for the parameter values $ a_1 = 1.5 $ and $ a_2 = 0.3 $.
The degree of non-convexity differs in different quadrants. 
Note in Fig.~\ref{fig:pen2} that
the contours of $ \psi $ resemble those of the separable non-convex penalty in Fig.~\ref{fig:pensep}, 
but the curvature is more pronounced in quadrants 2 and 4
and less pronounced in quadrants 1 and 3.
The parameters $ a_1 $ and $ a_2 $ determine the precise behavior of the penalty.

\begin{figure}[t]
	\centering
		\includegraphics{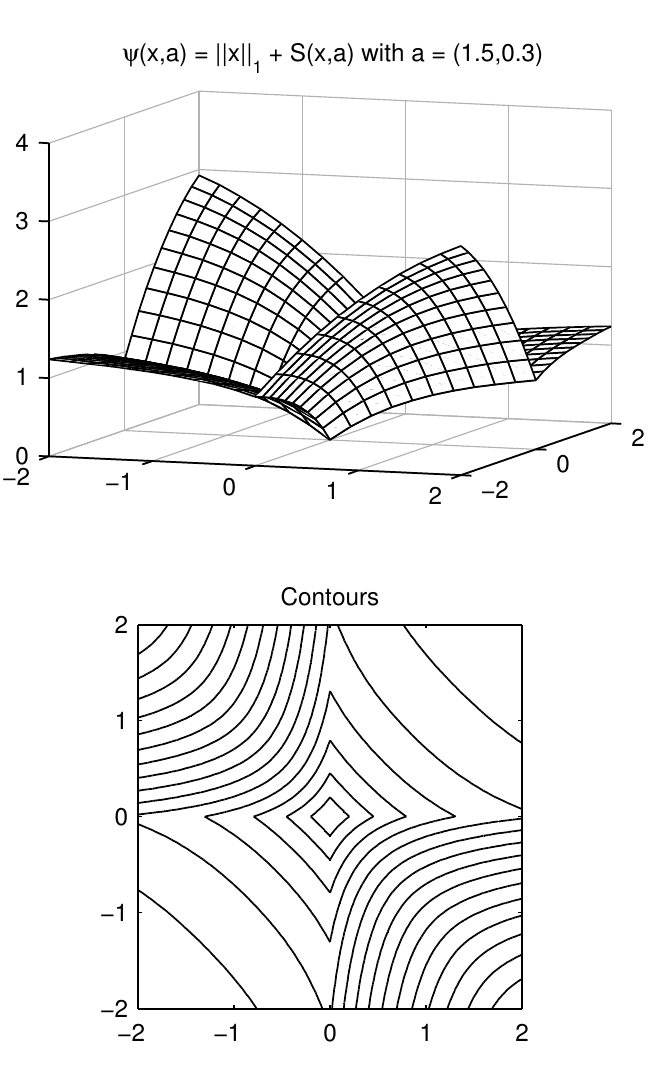}
	\caption{
		Non-separable non-convex penalty $ \psi $ in Definition \ref{def:bp}.
	}
	\label{fig:pen2}
\end{figure}

It is informative to consider special cases of the bivariate penalty.
If $ a_1 > 0 $ and $ a_2 = 0 $, 
then $ \psi $ simplifies to
\begin{equation}
	\label{eq:pensing}
	\psi(x) = \abs{x_1} + \abs{x_2} + \phi( x_1 + x_2 ; a_1/2) - \abs{ x_1 + x_2 }.
\end{equation}
If $ a_1 = a_2 $, then the penalty reduces to a separable function,
$ \psi(x) = \phi(x_1; a_1) + \phi(x_2; a_1) $
(see Fig.~\ref{fig:pensep}(b)).
If $ a_1 = a_2 = 0 $,
then it further reduces to the \la\ norm,
i.e., 
$ \psi(x) = \abs{x_1} + \abs{x_2} $
(see Fig.~\ref{fig:pensep}(a)).
In any case, 
if either $ a_1 $ or $ a_2 $ is positive, then $ \psi $ is non-convex. 

\medskip

The following theorem, based on Theorem \ref{thm:cvx2}, 
states how to restrict the parameters $ a_i $ 
to ensure problem \eqref{eq:biprob} is convex. 

\begin{theorem}
\label{thm:cvx3}
Let $ \psi \colon \RR^2 \to \RR $ be the bivariate penalty in Definition \ref{def:bp}.
Suppose $ H\tr H =  Q\tr \Gamma Q $
where $ Q $ and $ \Gamma $ are given by \eqref{eq:defQG}.
If $ a = (a_1, a_2) $ 
satisfy 
$	0 \le a_i \le \gamma_i / \lam $,
then the bivariate objective function $ f := \RR^2 \to \RR $, 
\begin{equation}
	\label{eq:biofun}
	f(x ; a ) = \half \norm{ y - H x }_2^2 + \lam \, \psi(x; a),
\end{equation}
is convex.

\end{theorem}

\begin{proof}
We write
\begin{equation}
	f(x ; a ) = g(x ; a ) +  \half y\tr y - y\tr H x + \lam \, \norm{ x }_1
\end{equation}
where $ g $ is given by \eqref{eq:defg} with $ K = H\tr H $ therein.
From Theorem \ref{thm:cvx2},  $ g $ is convex. 
Hence, $ f $ is convex 
because it is the sum of convex functions.
\end{proof}

Theorem \ref{thm:cvx3} gives a range for parameters $ a_1 $ and $ a_2 $
to ensure the objective function $ f $ is convex.
Precisely, the parameters should be bounded, respectively, by the eigenvalues of $ (1/\lam) H\tr H $.
To maximally induce sparsity, the parameters should be set to the maximal (critical) values, 
$ a_i = \gamma_i/\lam $.

Note that, even when the matrix $ H $ is singular (i.e., $ \gamma_1 = 0 $ or $ \gamma_2 = 0 $),
the bivariate penalty can be non-convex without spoiling the convexity
of the objective function $ f $
(if at least one of $ \gamma_i $ is positive). 
In other words, 
we need \emph{not} sacrifice the convexity of the objective function $ f $ 
in order to use sparsity-inducing non-convex penalties,
even when $ H $ is \emph{singular}.
This is an impossibility when the penalty $ \psi $ is a separable function.

Other bivariate penalties can be defined that ensure the objective function is convex;
however, the one defined here 
satisfies a further property we think should be required of a bivariate penalty. 
Namely, the proposed bivariate penalty lies between 
the two separable penalties corresponding to the 
minimum and maximum parameters $ a_i $.

\begin{theorem}
\label{thm:bounds}
Let $ a = (a_1, a_2) $ with $ a_i \ge 0 $.
Set
$ a_{\min} = \min\{ a_1, a_ 2 \} $
and
$ a_{\max} = \max\{ a_1, a_ 2 \} $.
The bivariate penalty $ \psi $ in Definition \ref{def:bp} satisfies
\begin{equation}
	\label{eq:bounds}
	\phi(x_1; a_{\max}) + \phi(x_2; a_{\max}) 
	\le \psi(x; a) \le 
	\phi(x_1; a_{\min}) + \phi(x_2; a_{\min}).
\end{equation}
\end{theorem}

The theorem is proven in Appendix \ref{sec:bounds}.
(The supplemental material has an animated illustration of these bounds.)

\begin{figure}[t]
	\centering
	\includegraphics[]{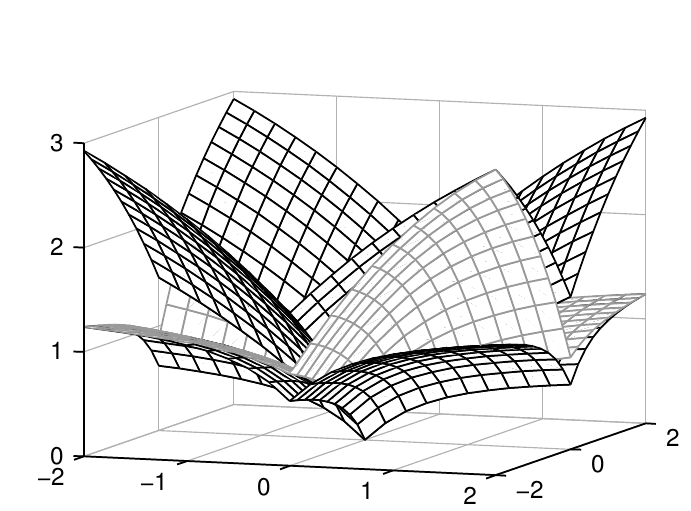}
	\caption{
		Illustration of Theorem \ref{thm:bounds}.
		The non-separable penalty $ \psi $ lies between two separable penalties.
	}
	\label{fig:sandwich}
\end{figure}

The inequality in Theorem \ref{thm:bounds} is tight: 
the lower and upper bounds are individually satisfied with equality on the lines $ \{ x = (t, t) \} $
and $ \{ x = (t, -t) \} $,
as illustrated in Fig.~\ref{fig:sandwich}.
Note that when $ a_1 = a_2 $ (i.e., $ a_{\min} = a_{\max} $), 
the theorem requires the penalty $ \psi $ to be separable. 
Indeed, the penalty \eqref{eq:defpsi} is separable when $ a_1 = a_2 $.

We think a bivariate penalty should satisfy the inequality of Theorem \ref{thm:bounds} for the following reason. 
In this work, we aim to induce pure sparsity (i.e., not structured sparsity, etc.).
Therefore, when $ H\tr H $ is a diagonal matrix
we should use a separable penalty. 
(A separable penalty best reflects an iid prior.)
It follows that when $ H\tr H = \gamma_1 I $,
the most suitable penalty (maintaining convexity of $ f $)
is the separable one:
$ \phi(x_1, \gamma_1/\lam) + \phi(x_2, \gamma_1/\lam) $.
A parameterized bivariate penalty should recover this
separable penalty as a special case. 
Moreover, 
if $ H\tr H = \gamma_2 I $ with $ \gamma_2 < \gamma_1 $, then
the most suitable penalty is again a separable one:
$ \phi(x_1, \gamma_2/\lam) + \phi(x_2, \gamma_2/\lam) $.
But $ \gamma_2 < \gamma_1 $
means the corresponding data fidelity term is less strongly convex
and thus the penalty term must be less strongly non-convex.
Consequently, we must have
\[
	 \phi(x_1, \gamma_1/\lam) + \phi(x_2, \gamma_1/\lam) 
	<
	 \phi(x_1, \gamma_2/\lam) + \phi(x_2, \gamma_2/\lam).
\]
When $ H\tr H $ 
has distinct eigenvalues $ \gamma_1 $ and $ \gamma_2 $,
the most suitable bivariate penalty should lie between these two separable penalties. 
Theorem \ref{thm:bounds} assures this.

If not suitably designed and utilized, it is conceivable that a non-separable penalty may
lead to correlation or structure in the estimated signal that is not present in original sparse signal.
To avoid unintentionally inducing correlation in the estimated signal, 
it seems reasonable that 
the bivariate penalty should
exhibit some similarity to the corresponding separable penalties
(that reflect unstructured sparsity).
Theorem \ref{thm:bounds} indicates the bivariate penalty \eqref{eq:defpsi} 
conforms to the relevant separable penalties.
It is still possible that some erroneous correlation might be introduced,
but such correlation is not evident in the experimental results.
We attribute this to Theorem \ref{thm:bounds}.

One may question the legitimacy of a method wherein penalty parameters
are set according to the data fidelity term.
Conventionally, the penalty term should reflect prior knowledge 
of the signal to be estimated;
it should not depend on $ H $, which represents the observation model.
(This is formalized in the Bayesian perspective 
where the objective function corresponds to a likelihood function
and the penalty term corresponds to a prior). 
The approach taken here, wherein the parameters
of the penalty term are based on properties of $ H $,
appears to violate this principle. 
However, 
the common practice of 
restricting the penalty to be convex also violates this principle. 
Probability densities 
(e.g.,
generalized Gaussian \cite{Mallat_1998},
mixture models \cite{Portilla_2003_TIP, Chipman_1997},
Bessel-K \cite{Fadili_2005_TIP}, and $\alpha$-stable \cite{Achim_2005_SPL}),
that accurately model sparsity,
 correspond to non-convex penalties. 
Using the \la\ norm as a penalty corresponds to the Laplace distribution (a relatively weak sparsity model).
The proposed bivariate sparse regularization (BISR) approach
is simply intended to follow a sparsity prior more closely.

\subsection{Separable penalties}

To clarify the value of non-separable regularization, 
we note a limitation of separable penalties. 

\begin{lemma}
\label{lemma:pensep}
Let the univariate penalty $ \phi \colon\RR \to \RR $ satisfy the properties of Sec.~\ref{sec:pen1}. 
The objective function $ f := \RR^2 \to \RR $, 
\begin{equation}
	f(x ; a ) = \half \norm{ y - H x }_2^2
		+ \lam \, \phi(x_1; a)
		+ \lam \, \phi(x_2; a),
\end{equation}
with $ \lam > 0 $ and $ a \ge 0 $, 
is convex only if 
\begin{equation}
	\phi''(0^+) \ge -(1/\lam) \min\{ \gamma_1, \gamma_2 \},
\end{equation}
or equivalently,
$	0 \le a \le \min \{ \gamma_1, \gamma_2 \} / \lam $,
where $ \gamma_i $ are the eigenvalues of $ H\tr H $.
\end{lemma}
\begin{proof}
Let $ u \in \RR^2 $ be an eigenvector of $ H\tr H $ corresponding to 
its minimum eigenvalue $ \gamma_{\min} $, i.e.,
\begin{equation}
	H\tr H u = \gamma_{\min} u.
\end{equation}
Consider $ f $ on a line in the direction of $ u $. 
Namely, 
define $ g \colon \RR \to \RR $ as
\begin{align}
	g(t) & = f( t u; a )
	\\
	& =
	\half \norm{ y - t H u }_2^2
		+ \lam \, \phi( t u_1 ;  a)
		+ \lam \, \phi( t u_2 ; a).
\end{align}
We will show that $ g $ is not convex when $ a > \gamma_{\min} $.
It will follow that $ f $ is not convex, 
because the restriction of a multivariate convex function to any line must also be convex. 

By the properties of $ \phi $, 
the function $ g $ is twice continuously differentiable on $ \RR_+ $
and its second derivative is given by
\begin{align}
	g''(t) 
	& =
	u H\tr H u 
	+ \lam u_1^2 \, \phi''( t u_1 ;  a)
	+ \lam u_2^2 \, \phi''( t u_2 ; a)
	\\
	& =
	\gamma_{\min} u\tr u
	+ \lam u_1^2 \, \phi''( t u_1 ;  a)
	+ \lam u_2^2 \, \phi''( t u_2 ; a).
\end{align}
Since $ \phi''( 0^+ ; a ) = -a $ is a defining property of $ \phi $,
we have
\begin{align}
	g''(0^+) & =
	\gamma_{\min} u\tr u
	- \lam u_1^2 \, a
	- \lam u_2^2 \, a
	\\
	& = 
	(\gamma_{\min} - \lam a) \, u\tr u.
\end{align}
Hence, convexity of $ g $ requires  $ \gamma_{\min} - \lam a \ge 0 $;
i.e., $ a \le \gamma_{\min} / \lam $.
\end{proof}

According to Lemma \ref{lemma:pensep}, a \emph{separable} non-convex penalty,
that ensures convexity of the objective function, is limited by the minimum eigenvalue of $ H\tr H $.
It cannot exploit the greater eigenvalue. 
This is unfavorable, because when one of the eigenvalues is close to zero, 
a separable penalty can be only mildly non-convex
and provides negligible improvement relative to the \la\ norm;
when $ H\tr H $ is singular,
we recover the \la\ norm. 
In contrast, 
a non-separable penalty can exploit both eigenvalues independently. 
Hence, 
non-separable penalties are most advantageous when the eigenvalues of $ H\tr H $ are quite different in value.

\section{Sparse Reconstruction}
\label{sec:recon}

Practical problems in signal processing involve far more than two variables.
Therefore, the proposed bivariate penalty \eqref{eq:defpsi} and convexity condition \eqref{eq:aicond}
are of little practical use on their own. 
In this section we show how they can be used to solve an $ N $-point 
linear inverse problem (with $ N > 2 $). 
We consider the problem of estimating a signal $ x \in \RR^N $ given $ y $,
\begin{equation}
	\label{eq:obsmod}
	y = H x + w
\end{equation}
where
$ H $ is a known linear operator,
$ x $ is known to be sparse, and
$ w $ is additive white Gaussian noise (AWGN).
We formulate the estimation of $ x $ as an optimization problem
with bivariate sparse regularization (BISR),
\begin{equation}
	\label{eq:defFdeconv}
	\hat x = \arg \min_{x \in \RR^N } \;
	\biggl\{
		F(x) =
		\half \norm{ y - H x }_2^2 + \frac{\lam}{2} \sum_n \psi( (x_{n-1}, x_n) ; a)
	\biggr\},
	\quad
	\lam > 0
\end{equation}
where 
$ a = (a_1, a_2) $
and $ \psi \colon \RR^2 \to \RR $ is the proposed bivariate penalty \eqref{eq:defpsi}.
In the penalty term, 
the first and last signal values pairs, $ ( x_0 , x_1 ) $ and $ ( x_{N}, x_{N+1} ) $, straddle
the end-points of $ x $.
As noted in Sec.~\ref{sec:notation}, 
we define $ x_n = 0 $ for $ n \notin \{ 1 , 2 , \dots, N \} $, 
which simplifies subsequent notation.

If $ a_1 = a_2 $,
then the bivariate penalty is separable,
i.e., 
$ \psi(u; a) = \phi(u_1; a_1) + \phi(u_2; a_1) $,
and the $ N $-point penalty term in \eqref{eq:defFdeconv}
reduces to $ \lam \sum_n \phi(x_n, a_1) $.
Hence,
we recover the standard (separable) formulation of sparse regularization.
In particular, 
if $ a_1 = a_2 = 0 $,
then
$ \psi( u ; 0 ) = \abs{ u_1 } + \abs{ u_2 } $
and 
the $ N $-point penalty term reduces to $ \lam \norm{ x }_1 $,
i.e., 
the classical sparsity-inducing convex penalty.

In order to induce sparsity more effectively,
we allow $ \psi $ to be non-separable;  i.e., $ a_1 \neq a_2 $.
To that end, the following section addresses the problem of how to set $ a_1 $ and $ a_2 $ 
in the bivariate penalty $ \psi $ to ensure convexity of the $ N $-variate objective function $ F $ in \eqref{eq:defFdeconv}.

\subsection{Convexity condition}

\begin{lemma}
\label{lemma:tridiag}
Let $ F \colon \RR^N \to \RR $ be defined in \eqref{eq:defFdeconv}
where
$ \psi \colon \RR^2 \to \RR $ is a parameterized bivariate penalty as defined in Definition \ref{def:bp}.
Let $ P $ be a positive semidefinite symmetric tridiagonal Toeplitz matrix, 
\begin{equation}
	\label{eq:defP}
	P = 
	\begin{bmatrix}
		p_0 & p_1 & & & \\
		p_1 & p_0 & p_1 & & \\
		& \ddots & \ddots & \ddots  \\
		& & p_1 & p_0 & p_1 \\
		& & & p_1 & p_0
	\end{bmatrix},
\end{equation}
such that
$ 0 \mle P \mle H\tr H $.
If the bivariate function $ f \colon \RR^2 \to \RR $ 
defined as
\begin{equation}
	\label{eq:defp}
	f(u) = \half u\tr
	\begin{bmatrix}
		p_0 & 2 p_1 \\
		2 p_1 & p_0
	\end{bmatrix}
	u + \lam \psi( u; a )
\end{equation}
is convex,
then $ F $ is convex.
\end{lemma}
The lemma is proven in Appendix \ref{sec:lemma:tridiag}. 
According to the lemma,
it is sufficient to restrict $ \psi $ so as to ensure convexity of the bivariate function $ f $ in \eqref{eq:defp}. 
Therefore, 
the allowed penalty parameters $ a_i $
can be determined from the tridiagonal matrix $ P $.
Using Theorem \ref{thm:cvx2} and Lemma \ref{lemma:tridiag}, we obtain Theorem \ref{thm:hrzo}.

\begin{theorem}
\label{thm:hrzo}
Let $ F \colon \RR^N \to \RR $ be defined in \eqref{eq:defFdeconv}
where
$ \psi \colon \RR^2 \to \RR $ is a parameterized bivariate penalty as defined in Definition \ref{def:bp}.
Let $ P $ be a symmetric tridiagonal Toeplitz matrix \eqref{eq:defP}
satisfying
$ 0 \mle P \mle H\tr H $.
If
\begin{equation}
	\label{eq:aipconds}
	0 \le a_1 \le (p_0 + 2 p_1)/\lam,
	\quad
	0 \le a_2 \le (p_0 - 2 p_1)/\lam,
\end{equation}
then $ F $ in \eqref{eq:defFdeconv} is convex.
\end{theorem}
\begin{proof}
Note the eigenvalue value decomposition 
\begin{equation}
	\begin{bmatrix}
		p_0 & 2 p_1 \\
		2 p_1 & p_0
	\end{bmatrix}
	=
	Q\tr
	\begin{bmatrix}
		p_0 + 2 p_1 & 0
		\\
		0 & p_0 - 2 p_1
	\end{bmatrix}
	Q	
\end{equation}
where $ Q $ is the orthonormal matrix \eqref{eq:defQG}.
Hence, $ f $ in \eqref{eq:defp} can be written as
\begin{equation}
	f(u) = 
	\frac{1}{2}
	u\tr 
	Q\tr
	\begin{bmatrix}
		p_0 + 2 p_1 & 0
		\\
		0 & p_0 - 2 p_1
	\end{bmatrix}
	Q
	u
	+
	\lam \,
	\psi( u;  a ).
\end{equation}
By Theorem \ref{thm:cvx2},
if  $ 0 \le a_1 \le (p_0 + 2 p_1)/\lam $
and $ 0 \le a_2 \le (p_0 - 2 p_1)/\lam $,
then $ f $ is convex.
It follows from Lemma \ref{lemma:tridiag} that $ F $ in \eqref{eq:defFdeconv} is convex.
\end{proof}

\subsection{Optimality condition}

In this section, we derive an explicit condition to 
verify the optimality of a prospective minimizer of
the objective function $ F $ in \eqref{eq:defFdeconv}.
The optimality condition is also useful for monitoring
the convergence of an optimization algorithm
(see the animation in the supplemental material). 

The general condition to characterize minimizers of a convex
function is expressed in terms of the subdifferential.
If $ F $ is convex, then $ x\opt \in \RR^N $ is a minimizer if and only if
$	0 \in \partial F( x \opt )  $
where $ \partial F $ is the subdifferential of $ F $. 

We seek an expression for the subdifferential of the objective function $ F $.
The function $ F $ in \eqref{eq:defFdeconv}
has a regularization term that is non-differentiable, non-convex, and non-separable. 
But with the aid of \eqref{eq:defpsi},
we may write the regularization term as:
\begin{align}
	\half \sum_{n} \psi( ( x_{n-1}  ,  x_{n} ) ; a ) 
	& = \half \sum_{n} \Bigl[ S( ( x_{n-1}  ,  x_{n} ) ; a )
			+ \norm{ ( x_{n-1}  ,  x_{n} ) }_1 \Bigr]
	\\
	& = \half \sum_{n} \Bigl[ S( ( x_{n-1}  ,  x_{n} ) ; a )
			+ \abs{ x_{n-1} } + \abs{ x_{n} } \Bigr]
	\\
	\label{eq:RS}
	& = \norm{ x }_1 
			+ \half \sum_{n} S( ( x_{n-1}  ,  x_{n} ) ; a )
\end{align}
where $ x_n = 0 $ for $ n \notin \{ 1 , 2 , \dots, N \} $.
We define
$ \SN \colon \RR^N \to \RR $ 
as
\begin{equation}
	\label{eq:defSN}
	\SN( x ; a ) = \half \sum_{n} S( ( x_{n-1}  ,  x_{n} ) ; a ) .
\end{equation}
The function $ \SN $ is differentiable, it being the sum of differentiable functions. 
Using \eqref{eq:RS}, we may express the objective function $ F $ in \eqref{eq:defFdeconv} as
\begin{equation}
	\label{eq:Fsimp}
	F(x) =
	\half \norm{ y - H x }_2^2
	+
	\lam \SN( x ; a )
	+
	\lam \norm{ x }_1.
\end{equation}
The benefit of \eqref{eq:Fsimp} 
compared to \eqref{eq:defFdeconv}
is that the regularization term (which is  non-differentiable, non-convex, and non-separable) 
is separated into 
a differentiable part and a convex separable part. 
The $ \SN $ term is differentiable and its gradient is easily evaluated.
The \la\ norm is separable and convex and its subdifferential is easily evaluated.

The gradient of $ \SN $ is given by
\begin{equation}
	\label{eq:SNgrad}
	[ \nabla \SN( x ; a ) ]_n
	=
	\half
	S_1( (x_{n}, x_{n+1}) ; a)
	+
	\half
	S_2( (x_{n-1}, x_n) ; a)
\end{equation}
where $ S_i $ is the partial derivative of $ S((x_1, x_2)) $ with respect to $ x_i $.
They are tabulated in \eqref{eq:defS1} and \eqref{eq:defS2}.

The subdifferential of the \la\ norm is separable \cite{Boyd_convex},
\begin{equation}
	\partial \norm{ x }_1 = \sign( x_1 ) \times \cdots \times \sign ( x_N )
\end{equation}
where $ \sign $ is the set-valued signum function
\begin{equation}
	\sign( t ) :=
	\begin{cases}	
		\{ 1 \}, & t > 0
		\\
		[-1, 1], \ & t = 0
		\\
		\{ -1 \}, & t < 0.
	\end{cases}
\end{equation}

Since the first two terms of \eqref{eq:Fsimp} are differentiable, the subdifferential of $ F $ is
\begin{equation}
	\partial F(x) = 
	H\tr ( H x - y )
	+
	\lam  \nabla \SN( x ; a )
	+
	\lam \partial \norm{ x }_1.
\end{equation}
Hence the condition $ 0 \in \partial F(x\opt) $
can be expressed as
\begin{equation}
	( 1 / \lam )
	H\tr ( y - H x\opt  )
	-
	\nabla \SN( x\opt ; a )
	\in
	\partial \norm{ x\opt }_1.
\end{equation}
Expressing this condition component-wise,
we have the following result.

\begin{theorem}
\label{thm:optim}
If $ a = (a_1, a_2) $
is chosen so that 
the objective function $ F $ in \eqref{eq:defFdeconv}
is convex, 
then $ x\opt $ minimizes $ F $
if and only if
\begin{equation}
	\label{eq:optcond}
	\frac{ 1 }{ \lam }
	[ H\tr ( y - H x\opt  ) ]_n
	-
	[ \nabla \SN( x\opt ; a ) ]_n
	\in
	\sign ( x\opt_n ),
	\ \ 
	n = 1, \dots, N.
\end{equation}
\end{theorem}
This condition can be depicted using a scatter plot
as in Fig.~\ref{fig:deconv1_scatter} below. 
The points in the scatter plot 
show the left-hand-side of \eqref{eq:optcond} versus
 $ x_n $ for $ n = 1, \dots, N $.
A signal $ x\opt $ is a minimizer of $ F $ if and only if
the points in the scatter plot lie on the graph of the set-valued signum function
(e.g., Fig.~\ref{fig:deconv1_scatter}).

\subsection{Sparse Deconvolution}
\label{sec:deconv}

We apply Theorem \ref{thm:hrzo} to the sparse deconvolution problem. 
In this case, the linear operator $ H $ represents convolution,
i.e., 
\begin{equation}
	[H x]_n = \sum_k h_{n - k} \, x_{ k }.
\end{equation}
That is, $ H $ is a Toeplitz matrix.
It represents a linear time-invariant (LTI) system 
with frequency response given by
the Fourier transform of $ h $,
\begin{equation}
	\label{eq:Hom}
	H( \om ) =  \sum_n h_n \, \myE^{ - \myJ \om n}.
\end{equation}
Similarly, 
the matrix $ P $ in \eqref{eq:defP} represents
an LTI system with a real-valued frequency response,
\begin{align}
	P( \om )
	& = p_1 \myE^{-\myJ \om} + p_0 + p_1  \myE^{\myJ \om} 
	\\
	\label{eq:Pcos}
	& = p_0 + 2 p_1 \cos(\om).
\end{align}
Specializing Theorem \ref{thm:hrzo} to the problem of deconvolution, 
we have the following result. 

\begin{theorem}
\label{thm:filt}
Let $ H $ in \eqref{eq:defFdeconv} represent convolution.
Let $ P $ in \eqref{eq:Pcos} satisfy
\begin{equation}
	\label{eq:ZPH}
	0 \le P( \om) \le \abs{ H( \om ) }^2,
	\ \ 
	\forall \om,
\end{equation}
where $ H(\om) $ is given by \eqref{eq:Hom}.
If  $ 0 \le a_1 \le P(0)/\lam $
and $ 0 \le a_2 \le P(\pi)/\lam $,
then $ F $ in \eqref{eq:defFdeconv} is convex.
\end{theorem}

\begin{proof}
Using the convolution property of the discrete-time Fourier transform, condition \eqref{eq:ZPH} is equivalent to $ 0 \mle P_{\infty} \mle H_{\infty}\tr  H_{\infty} $,
where these are doubly-infinite Toeplitz matrices corresponding to discrete-time signals defined on $ \ZZ $.
Since any principal sub-matrix of a positive semidefinite matrix is also positive semidefinite, 
the inequality is also true for finite matrices $ P $ and $ H\tr H $,
which can be recognized as principal sub-matrices of the corresponding doubly-infinite matrices.  
Therefore, $ 0 \mle P \mle H\tr H $,
and by Theorem \ref{thm:hrzo},
the objective function $ F $ is convex
if $ a_i $ satisfy \eqref{eq:aipconds}.
Noting that $ P(0) = p_0 + 2 p_1 $ and $ P(\pi) = p_0 - 2 p_1 $
yields the result.
\end{proof}

To induce sparsity as strongly as possible, 
$ P( \om ) $ should be as close as possible to the upper bound $ \abs{ H( \om ) }^2 $.
The determination of $ P( \om ) $ satisfying such constraints can
be efficiently and exactly performed using semi-definite programming (SDP)
as described by Dumitrescu \cite{Dumitrescu_2007}. 
Since $ P $ is low-order here, the SDP computation is negligible. 

An example is illustrated in Fig.~\ref{fig:filter1}.
The impulse response $ h $
is shown in Fig.~\ref{fig:filter1}(a).
The square magnitude of the frequency response $ \abs{ H( \om ) }^2 $
is shown in Fig.~\ref{fig:filter1}(b).
The frequency response 
$
	P ( \om ) = 0.4 + 0.2 \cos(\om) 
	$
is real-valued, non-negative,
and approximates $ \abs{ H( \om ) }^2 $ from below.
According to Theorem \ref{thm:filt}, the objective function $ F $ is
convex if
$ 0 \le a_1 \le 0.6/ \lam $
and
$ 0 \le a_2 \le 0.2/ \lam $.
This filter will be used in Example 1 below (Sec.~\ref{sec:EG1}).

\begin{figure}[t]
	\centering
		\includegraphics{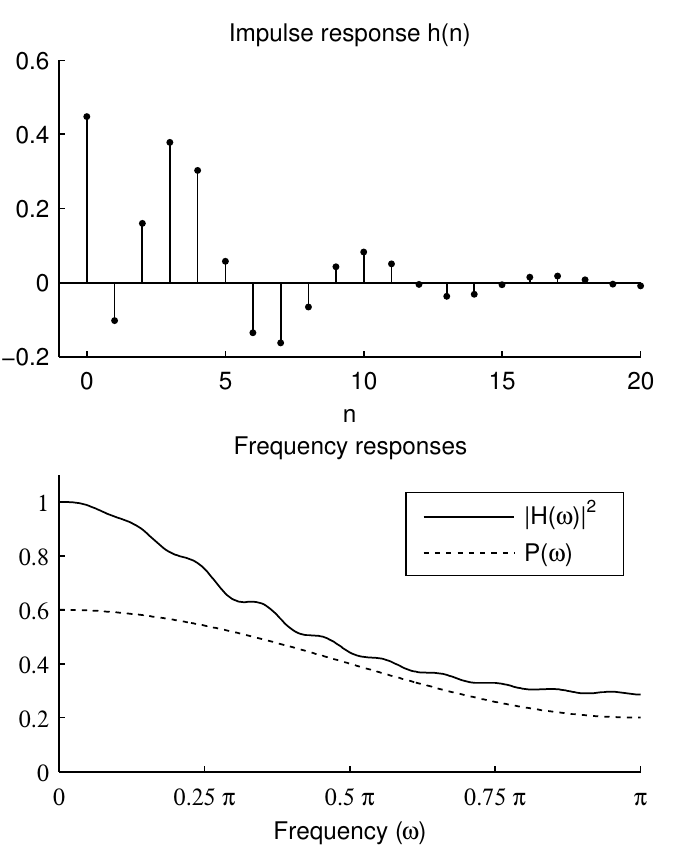}
	\caption{
		Filters $ H(\om) $ and  $ P(\om) $ for Example 1.
	}
	\label{fig:filter1}
\end{figure}

\begin{figure}[t]
	\centering
		\includegraphics{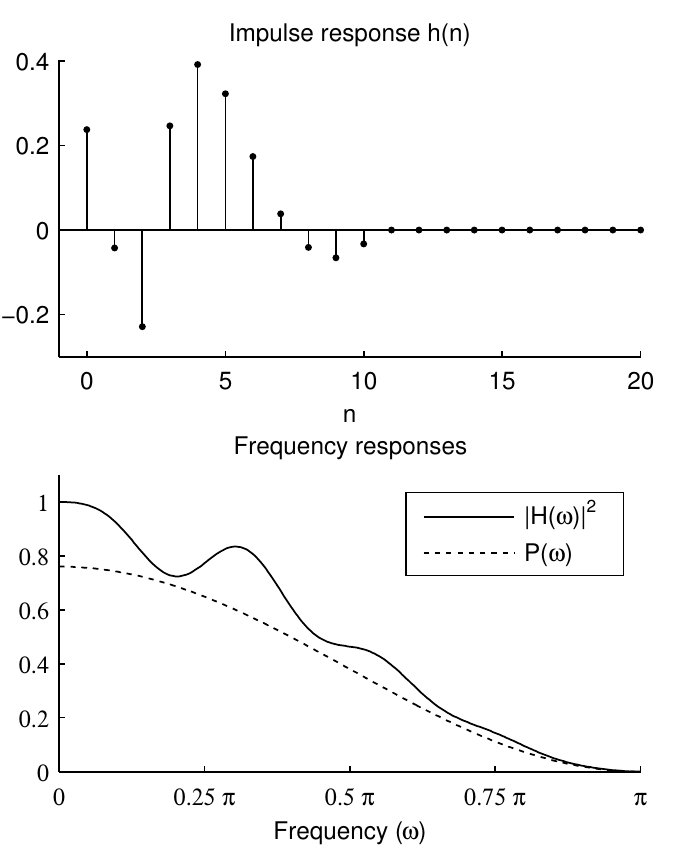}
	\caption{
		Filters $ H(\om) $ and  $ P(\om) $ for Example 2.
	}
	\label{fig:filter2}
\end{figure}

Another example is illustrated in Fig.~\ref{fig:filter2}. 
The frequency response $ H( \om ) $ has a null at $ \om = \pi $.
Hence, the system $ H $ is not invertible.
Since $ H(\pi) = 0 $,
any $ P $ satisfying \eqref{eq:ZPH} also has $ P(\pi) = 0 $.
We find $ P( \om ) = 0.38 ( 1 + \cos \om ) $ satisfies \eqref{eq:ZPH};
see Fig.~\ref{fig:filter2}(b).
Therefore, 
according to Theorem~\ref{thm:filt}, the objective function $ F $ is
convex if $ 0 \le a_1 \le 0.76/ \lam $ and $ a_2 = 0 $.
For $ \{ a_1 > 0, a_2 = 0 \} $, the multivariate penalty is non-convex and non-separable. 
For $ \{ a_1 = a_2 = 0 \} $, the penalty is simply the \la\ norm
(convex and separable).
This filter will be used in Example 2 below (Sec.~\ref{sec:EG2}).

\bigskip

In reference to the filters $ H $ illustrated in Figs.~\ref{fig:filter1} and \ref{fig:filter2},
it is informative to consider the case of a separable
penalty. 
If $ \psi $ is a separable penalty (i.e., $ a_1 = a_2 $), 
then the objective function $ F $ is
convex only if $ 0 \le a_1 = a_2 \le \min_{\om}  \abs{ H(\om) }^2 $.
For the filter of Fig.~\ref{fig:filter1} this leads to the constraint
 $ 0 \le a_1 = a_2 \le 0.26 $,
 meaning that the separable penalty may be non-convex (i.e., $ a_1 = a_2 > 0 $ is allowed).   
On the other hand, 
for the filter of Fig.~\ref{fig:filter2}
this leads to the constraint  $ a_1 = a_2 = 0 $,
i.e., $ \psi(x, a) = \abs{ x_1 } + \abs{ x_2 } $.
When the filter $ H $ is not invertible, 
the only non-convex penalties maintaining convexity of the objective function $ F $
are \emph{non-separable} penalties. 
Since inverse problems often involve singular or nearly singular operators $ H $,
this motivates the development of non-separable penalties as herein.

\section{Algorithms}

\subsection{Iterative L1 minimization }

We present an iterative L1 norm minimization algorithm
to solve \eqref{eq:defFdeconv}.
The derivation is based on majorization-minimization (MM)
\cite{FBDN_2007_TIP}.
The MM principle consists of the iteration
\begin{equation}
	\label{eq:defMM}
	x\iter{k+1} \in \arg\min_{x} F\maj (x; x\iter{k})
\end{equation}
where
$ k $ is the iteration index  and
$ F\maj $ denotes a majorizer of the objective function $ F $, i.e., 
\begin{align}
	F\maj (x; v) & \ge F(x), \  \ \text{for all $x$, $v$}
	\\
	F\maj (v; v) & = F(v), \ \ \text{for all $v$.}
\end{align}

\begin{figure}
	\centering
	\includegraphics{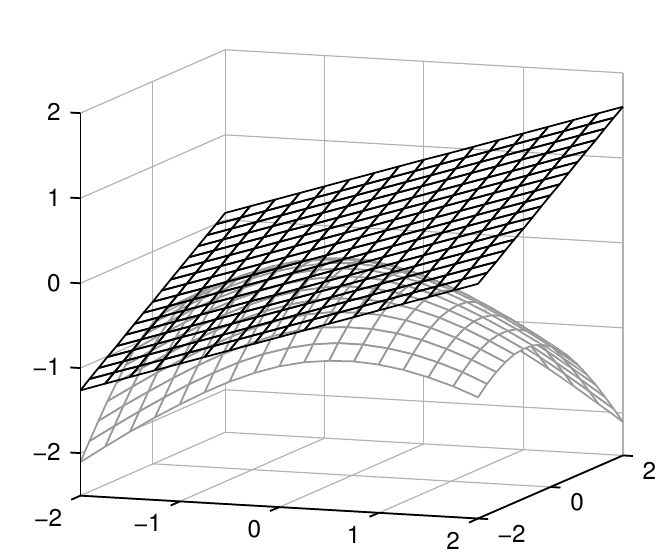}
	\caption{
		Majorization of concave function $ S $ by linear function $ S\maj $. 
	}
	\label{fig:linmaj}	
\end{figure}

The function $ S \colon \RR^2 \to \RR $ defined in Lemma \ref{lemma:defS}
is twice continuously differentiable and concave.
Therefore, a majorizer of $ S $ 
is given by
\begin{equation}
	S\maj (x; v) = S(v) + [\nabla S(v)]\tr (x - v)
\end{equation}
where $ \nabla S $ is the gradient of $ S $.
See Fig.~\ref{fig:linmaj}.
The majorizer $ S\maj (x; v) $ is linear in $ x $; hence convex in $ x $.
A majorizer of the bivariate penalty $ \psi $ 
is then obtained using \eqref{eq:defpsi}, 
\begin{align}
	\psi\maj (x; v) & = S\maj(x; v) + \norm{ x }_1
	\\
	& =  S(v) + [\nabla S(v)]\tr (x - v) + \norm{ x }_1.
\end{align}
The majorizer $ \psi\maj (x; v) $ is convex in $ x $, it being a sum of convex functions.

Recall the objective function $ F $ in \eqref{eq:defFdeconv} 
can be expressed as \eqref{eq:Fsimp}.
Therefore, a majorizer of the objective function $ F $ in \eqref{eq:defFdeconv} is given by
\begin{equation}
	F\maj(x; v) = 
	\half \norm{ y - H x }_2^2
	+
	\lam \SN\maj( x ; v ; a )
	+
	\lam \norm{ x }_1
\end{equation}
where
$ \SN\maj( x ; v ; a ) $ is a majorizer of $ \SN( x ; a ) $.
We recall the bivariate function $ S $
is concave and twice continuously differentiable.
Therefore, the $ N $-variate function $ \SN $ 
is concave as it is a sum of concave functions.
Likewise, $ \SN $ is twice continuously differentiable. 
Therefore, a majorizer of $ \SN $ is given by
\begin{equation}
	\SN\maj (x; v; a ) = \SN(v ; a) + [\nabla \SN(v ; a )]\tr (x - v)
\end{equation}
where $ \nabla \SN $ is given by \eqref{eq:SNgrad}.
Hence, $ F\maj $ is given by
\begin{equation}
	\label{eq:defFM}
	F\maj(x; v)
	= 
		\half \norm{ y - H x }_2^2
		+
		\lam
		 [\nabla \SN(v ; a )]\tr x
		+ 
		\lam \norm{ x }_1
		+ C(v)
\end{equation}
where $ C(v) $ does not depend on $ x $.
The MM iteration \eqref{eq:defMM} is then given by
\begin{equation}
	\label{eq:defMMF}
	x\iter{k+1} \in \arg\min_{ x \in \RR^N } 
		\Bigl\{
		\half \norm{ y - H x }_2^2
		+
		\lam
		 [\nabla \SN( x\iter{k} ; a )]\tr x
		+ 
		\lam \norm{ x }_1
		\Bigr\}.
\end{equation}
This is a standard \la\ norm optimization problem which can be solved
by several methods (e.g. proximal methods, ADMM).
Hence, the solution to \eqref{eq:defFdeconv} can be 
obtained by iterative \la\ norm minimization. 

Based on the theory of MM algorithms,
it is guaranteed that $ x\iter{k} $ converges to the minimizer of $ F $ 
when $ F $ is strictly convex
\cite{Jacobson_2007_TIP, Byrne_2014_IterOpt}.
If $ H $ is not invertible, then $ F $ may be convex without being strictly convex.
In this case:
(a) the function value $ F(x\iter{k}) $ converges to the minimum value of $ F $,
and
(b) if $ F $ has a unique minimizer, then $ x\iter{k} $ converges to it.
See Theorems 4.1 and 4.4 of 
\cite{Jacobson_2007_TIP}.

\subsection{Iterative thresholding}
\label{sec:ISTA}

We present an iterative thresholding algorithm to solve \eqref{eq:defFdeconv}.
The algorithm is an immediate application of forward-backward splitting (FBS) \cite{Combettes_2005, Combettes_2011_chap}.
The FBS algorithm minimizes a function of the form $ f_1 + f_2 $ where
both $ f_1 $ and $ f_2 $ are convex and additionally
$ \nabla f_1 $ is Lipschitz continuous.
To apply the FBS algorithm to problem \eqref{eq:defFdeconv}, 
we express $ F $ using \eqref{eq:Fsimp}.
The first two terms of \eqref{eq:Fsimp} constitute the smooth convex function $ f_1 $.
We remark that since $ \Theta $ is concave, the Lipschitz constant of $ \nabla f_1 $ is bounded
by $ \rho $ where $ \rho $ is the maximum eigenvalue of $ H\tr H $.
The $ \ell_1 $ norm term in \eqref{eq:Fsimp} constitutes the (non-smooth) convex function $ f_2 $.
An FBS algorithm to solve problem \eqref{eq:defFdeconv} is then given by
\begin{subequations}
\label{eq:defMM_ISTA}
\begin{align}
	z\iter{k} & = 
	x\iter{k} +  \mu \Bigl[  H\tr (y - H x\iter{k}) - \lam \nabla \SN(x\iter{k} ; a )  \Bigr]
	\\
	x\iter{k+1} & = \soft( z\iter{k} , \mu \lam )
\end{align}
\end{subequations}
where
$ 0 < \mu < 2 / \rho $
where $ \rho $ is the maximum eigenvalue of $ H\tr H $.
The parameter $ \mu $ can be viewed as a step-size.
The soft thresholding function
\begin{equation}
	\soft(t , T)
	:= 
	\begin{cases}
		t - T,  \ \ & t \ge T
		\\
		0, & \abs{ t } \le T
		\\
		t + T, & t \le -T
	\end{cases}
\end{equation}
is applied element-wise to vector $ z\iter{k} $.
As an FBS algorithm, it is guaranteed that $ x\iter{k} $ converges to a minimizer of $ F $.

This algorithm resembles the classical
iterative shrinkage/thresholding algorithm (ISTA) \cite{DDDM_2003, Fig_2003_TIP}.
Note that ISTA was derived in \cite{DDDM_2003, Fig_2003_TIP} using 
MM and was shown to converge for $ 0 < \mu < 1/\rho $.
However, 
the same algorithm derived using FBS is known to converge for twice this step size.
This is a practical advantage because the larger step-size generally yields
faster convergence of the algorithm.
We implement the algorithm with $ \mu = 1.9 / \rho $.

We further note that the iterative thresholding algorithm \eqref{eq:defMM_ISTA}
has the property that $ F( x\iter{k} ) $ monotonically decreases.
For 
$ 0 < \mu < 1/\rho $, this monotonic decreasing property follows from the MM-based derivation.
For $ 0 < \mu < 2/\rho $,
the proximal theory of FBS \cite{Combettes_2005, Combettes_2011_chap}
ensures convergence but not the monotonic decreasing property.
However, for this larger range of $ \mu $
the algorithm does in fact have
the monotonic decreasing property \cite{Bayram_2015_NCISTA_arxiv, She_2009_EJS}.

\section{Numerical Examples}
\label{sec:EGS}

\subsection{Example 1}
\label{sec:EG1}

\begin{figure}[t]
	\centering
		\includegraphics[scale = 0.85]{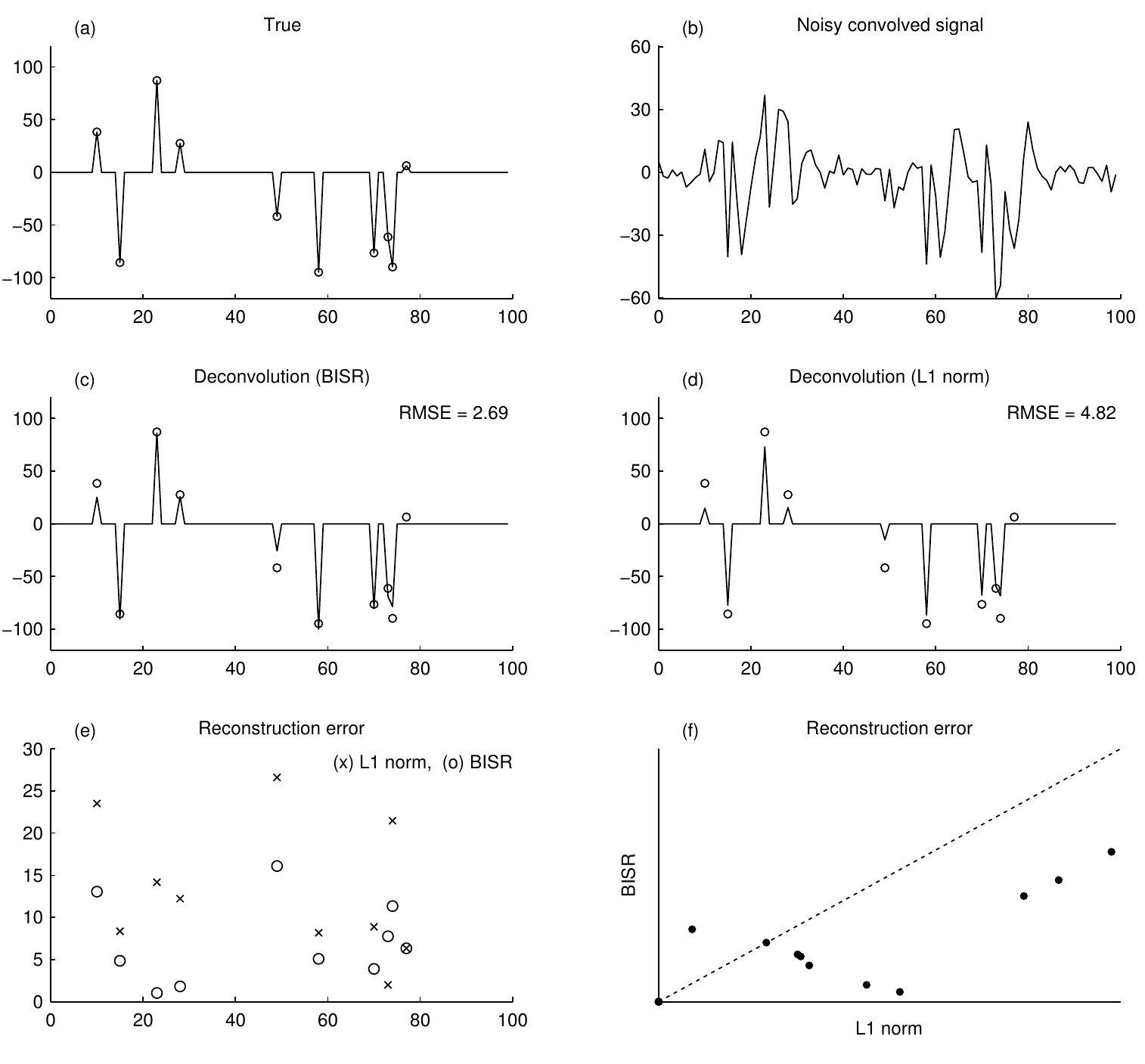}
	\caption{
		Example 1 of sparse deconvolution
		using bivariate sparse regularization (BISR).
	}
	\label{fig:deconv1}
\end{figure}

The 100-point sparse signal illustrated in Fig.~\ref{fig:deconv1}(a)
is convolved with the impulse response $ h $ shown in Fig.~\ref{fig:filter1}.
The convolved signal is corrupted by additive white Gaussian noise (AWGN)
with standard deviation $ \sigma = 4 $.
The corrupted signal is shown in Fig.~\ref{fig:deconv1}(b).
To perform sparse deconvolution using BISR,
we need to define the univariate penalty $ \phi $ and
set the regularization parameters $ \lam $ and $ a = (a_1, a_2) $
in the objective function \eqref{eq:defFdeconv}.
For $ \phi $, 
we use the arctangent penalty in Table~\ref{table:penalties}.
We set $ \lam $ straightforwardly as
\begin{equation}
	\label{eq:lamrule}
	\lam = \beta \sigma \norm{ h }_2 
\end{equation}
which follows from an analysis of
sparse optimality conditions
\cite{Fuchs_2007_JSTSP, Fuchs_2009_sysid, Selesnick_2014_TSP_MSC}. 
We set $ \beta = 2.5 $, similar to the `three sigma' rule. 
This choice of $ \beta $ is not intended to minimize the mean square error,
but rather to inhibit false impulses appearing in the estimated sparse signal.

To set $ a $, we use the non-negative function $ P(\om) $ 
shown in Fig.~\ref{fig:filter1}.
This filter has $ P(0) = 0.6 $ and $ P(\pi) = 0.2 $.
Therefore, according to Theorem \ref{thm:filt},
the objective function is convex if
$ 0 \le a_1 \le 0.6/\lam $
and 
$ 0 \le a_2 \le 0.2/\lam $.
To maximally induce sparsity, we set $ a_i $ to their respective maximal values. 
To perform deconvolution using BISR (i.e., to minimize the objective function $ F $)
we use the iterative thresholding algorithm (Sec.~\ref{sec:ISTA})
with a step-size of $ \mu = 1.9 / \rho $.
We run the algorithm until a stopping condition is satisfied.
As a stopping condition, we use
$
	\norm{ x\iter{k+1} - x\iter{k} }_\infty 
	\le 10^{-4}
	\times
	\norm{ x\iter{k} }_\infty
$
where $ k $ is the iteration index.
The run-time (averaged over 50 realizations) is about 8.8 milliseconds
using a 2013 MacBook Pro (2.5 GHz Intel Core i5) running Matlab R2011a.

The BISR solution is shown in Fig.~\ref{fig:deconv1}(c).
It has a root-mean-square-error (RMSE) of 2.7,
about 56\% that of the \la\ norm solution,
shown in Fig.~\ref{fig:deconv1}(d) for comparison.
Figure \ref{fig:deconv1}(e) shows the reconstruction error of both solutions.
The relative accuracy of the BISR solution is further illustrated in the scatter plot of Fig.~\ref{fig:deconv1}(f),
which shows the error of the two solutions plotted against each other.
Most of the points lie below the diagonal line, meaning that the
BISR solution has less error than the \la\ norm solution. 

\begin{figure}
	\centering
		\includegraphics{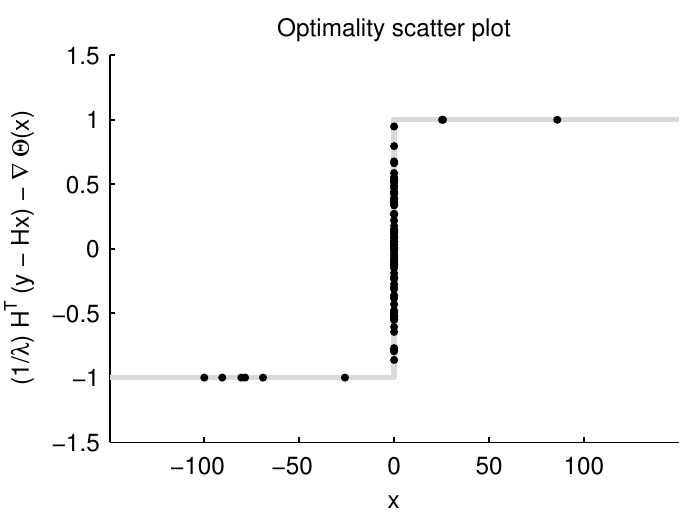}
	\caption{
		Optimality condition for Example 1.
	}
	\label{fig:deconv1_scatter}
\end{figure}

We verify the optimality of the obtained BISR solution using \eqref{eq:optcond}
as illustrated by the scatter plot in Fig.~\ref{fig:deconv1_scatter}.
The obtained solution is confirmed to be a global minimizer of the objective function
because
the points in the scatter plot lie on the graph of the signum function.

We compare the proposed BISR method with several other algorithms
in Table~\ref{table:rmse1} and Fig.~\ref{fig:deconv1_rmse}.
For the comparison, 
we vary the noise standard deviation
$ \sigma $ and generate 200 sparse signals and noise realizations
for each value of $ \sigma $.
Each sparse signal consists of 10 randomly located impulses 
with amplitudes uniformly distributed between $-$100 and 100. 
Table~\ref{table:rmse1} gives the average RMSE and run-time
(with all algorithms implemented in Matlab on the same computer).

\begin{table}
	\caption{\label{table:rmse1}
		Average RMSE and run-time for Example 1}
	\begin{center}
	\begin{tabular}{@{} l ccccc r@{}} 
		\toprule 
Algorithm   & $ \sigma\!=\!1 $ & $ \sigma\!=\!2 $  & $ \sigma\!=\!4 $  & $ \sigma\!=\!8 $  & $ \sigma\!=\!16 $   & (msec)    \\ 
		\midrule 
L1              & 1.17       & 2.32       & 4.43       & 8.19       & 13.47    &  3.5  \\ 
L1+debiasing    & 0.62       & 1.26       & 2.57       & 5.46       & 11.92 &   3.6   \\ 
Lp (p = 0.5)    & 0.66       & 1.19       & 2.39       & 5.12       & 12.05    & 4.3  \\ 
SBR (L0)        & 0.65       & 1.15       & 2.38       & 5.33       & 15.35   &  2.1 \\ 
IMSC            & 0.50       & 1.00       & 2.23       & 5.00       & 11.02   &  323.9  \\ 
IPS             & 0.51       & 1.02       & 2.23       & 4.89       & 10.96   &    5.1 \\ 
BISR (log)      & 0.52       & 1.11       & 2.53       & 5.62       & 11.58  &  8.2  \\ 
BISR (rat)      & 0.51       & 1.06       & 2.41       & 5.42       & 11.46   &  8.6 \\ 
BISR (atan)     & 0.50       & 1.03       & 2.30       & 5.13       & 11.22  &  8.8  \\ 
	\bottomrule 
	\end{tabular} 
\end{center}
\end{table}

\begin{figure}
	\centering
	\includegraphics{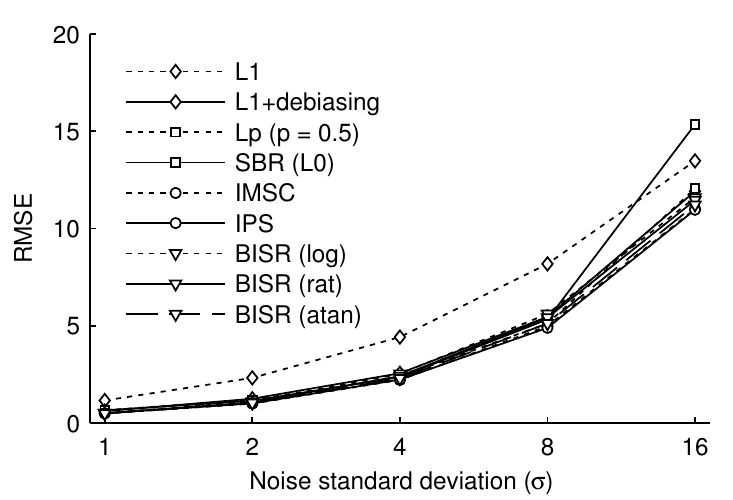}
	\caption{Average RMSE for Example 1.}
	\label{fig:deconv1_rmse}
\end{figure}

Included in the comparison are the following methods:
$ \ell_1 $ norm regularization, 
$ \ell_1 $ norm regularization with debiasing, 
$ \ell_p $ pseudo-norm regularization with $ p $ of $ 0.5 $, 
the \emph{single best replacement} (SBR) algorithm \cite{Soussen_2011_TSP},
\emph{iterative maximally sparse convex} (IMSC) regularization \cite{Selesnick_2014_TSP_MSC},
and 
the \emph{iterative p-shrinkage} (IPS) algorithm \cite{Voronin_2013_ICASSP, Woodworth_2015_arXiv}.

Each method provides an improvement over $ \ell_1 $ norm regularization. 
The first approach is $ \ell_1 $ norm regularization with debiasing.
Debiasing is a post-processing step that applies unbiased least squares approximation
to re-estimate the non-zero amplitudes \cite{Figueiredo_2007_GPSR}.
This method solves the systematic underestimation of non-zero amplitudes
from which $ \ell_1 $ norm regularization suffers, 
but it is still influenced by noise in the observed data.
As shown in Table~\ref{table:rmse1}, several other methods 
generally perform better than $ \ell_1 $ regularization with debiasing. 

Non-convex regularization using the $ \ell_p $ pseudo-norm with 
$ 0 < p < 1 $ is also a common approach for improving upon $ \ell_1 $ norm regularization. 
We use $ p = 0.5 $ with $ \lambda $ set according to \eqref{eq:lamrule}
with $ \beta = 8 $, a value we found worked well on average for this example.
This method performed similarly as $ \ell_1 $ norm regularization with debiasing. 

For non-convex regularization using the $ \ell_0 $ pseudo-norm
we use the SBR algorithm \cite{Soussen_2011_TSP}.
A comparison of several methods for $ \ell_0 $ pseudo-norm regularization
showed the SBR algorithm to be state-of-the-art \cite{Selesnick_2014_TSP_MSC}.
We use SBR with $ \lam $ set according to \eqref{eq:lamrule} with $ \beta = 70 $, which we found 
worked well here except for high noise levels.
We were unable to find a single $ \beta $ value that worked well
over the range of $ \sigma $ considered here. 
The SBR algorithm performed similarly to $ \ell_p $ pseudo-norm regularization, 
except for the highest noise level of $ \sigma = 8 $.

The IMSC algorithm for sparse deconvolution proceeds by solving 
a sequence of convex sub-problems \cite{Selesnick_2014_TSP_MSC}. 
The regularization term of each sub-problem is individually designed to 
be maximally non-convex (i.e., sparsity-inducing).
In contrast to the proposed BISR approach,
IMSC does not solve a prescribed optimization problem --- each iteration yields a new convex problem to be solved.
The formulation of each problem in IMSC requires the solution to a semidefinite problem (SDP);
therefore, the IMSC approach is very slow.
(IMSC is up to 100 times slower than other algorithms; see Table \ref{table:rmse1}.)
In IMSC, the parameter $ \lam $ can be set the same as for $ \ell_1 $ deconvolution;
hence, we again set $ \lam $ using \eqref{eq:lamrule} with $ \beta = 2.5 $. 
From Table~\ref{table:rmse1}, IMSC performs better than the preceding methods.

The IPS algorithm \cite{Voronin_2013_ICASSP, Woodworth_2015_arXiv}
generalizes the classic iterative shrinkage-thresholding algorithm (ISTA) \cite{DDDM_2003, Fig_2003_TIP}.
The IPS algorithm replaces soft-thresholding in ISTA by a 
threshold function that does not underestimate large-amplitude signal values. 
Notably, IPS can be understood as a method that seeks to minimize a 
prescribed (albeit implicit) non-convex objective function;
furthermore, each iteration of IPS can be understood as the exact minimization
of a convex problem.
Moreover, 
as IPS is an iterative thresholding algorithm, it is computationally very efficient.
For this example, IPS gives excellent results when
the parameter $ \lam $ is set using \eqref{eq:lamrule} with $ \beta = 2.5 $. 
It performs about the same as IMSC but runs much faster.

To complete the comparison, we compare the proposed BISR method 
using the three penalties listed in Table \ref{table:penalties}. 
(The BISR method can be used with any penalty function satisfying the properties
listed at the beginning of Sec.~\ref{sec:pen1}.)
Referring to Table~\ref{table:rmse1}, 
the three penalties give about the same result at low noise levels;
at higher noise levels, the arctangent penalty tends to perform better than the other two penalties.
Compared to the other algorithms, 
BISR with the arctangent penalty does about as well as IPS and IMSC
at low noise levels, but not quite as well at higher noise levels.
Note that the proposed BISR approach is the only algorithm
minimizing a prescribed convex objective function.

\subsection{Example 2}
\label{sec:EG2}

This is like Example 1, except we use the convolution filter $ H $ shown in Fig.~\ref{fig:filter2}.
This filter is not invertible. 
The sparse signal (Fig.~\ref{fig:deconv2}(a))
is convolved with the impulse response $ h $
and corrupted by AWGN ($ \sigma = 4 $).
We set $ \lam $ using \eqref{eq:lamrule} with $ \beta = 2.5 $ as in Example 1.
Since the filter $ P $ has $ P(0) = 0.76 $ and $ P(\pi) = 0 $,
Theorem \ref{thm:filt} states that
the objective function \eqref{eq:defFdeconv} is convex if
$ 0 \le a_1 \le 0.76/\lam $
and 
$ a_2 = 0  $.
We run the algorithm of Sec.~\ref{sec:ISTA} 
with the same stopping condition
to obtain the BISR solution  (Fig.~\ref{fig:deconv2}(c)).
The BISR solution has an RMSE of 1.8, which is about 50\% that of the
\la\ norm solution (Fig.~\ref{fig:deconv2}(d)).
As in Example~1, the optimality of the obtained BISR solution can be readily verified.

\begin{figure}[t]
	\centering
		\includegraphics[scale = 0.85]{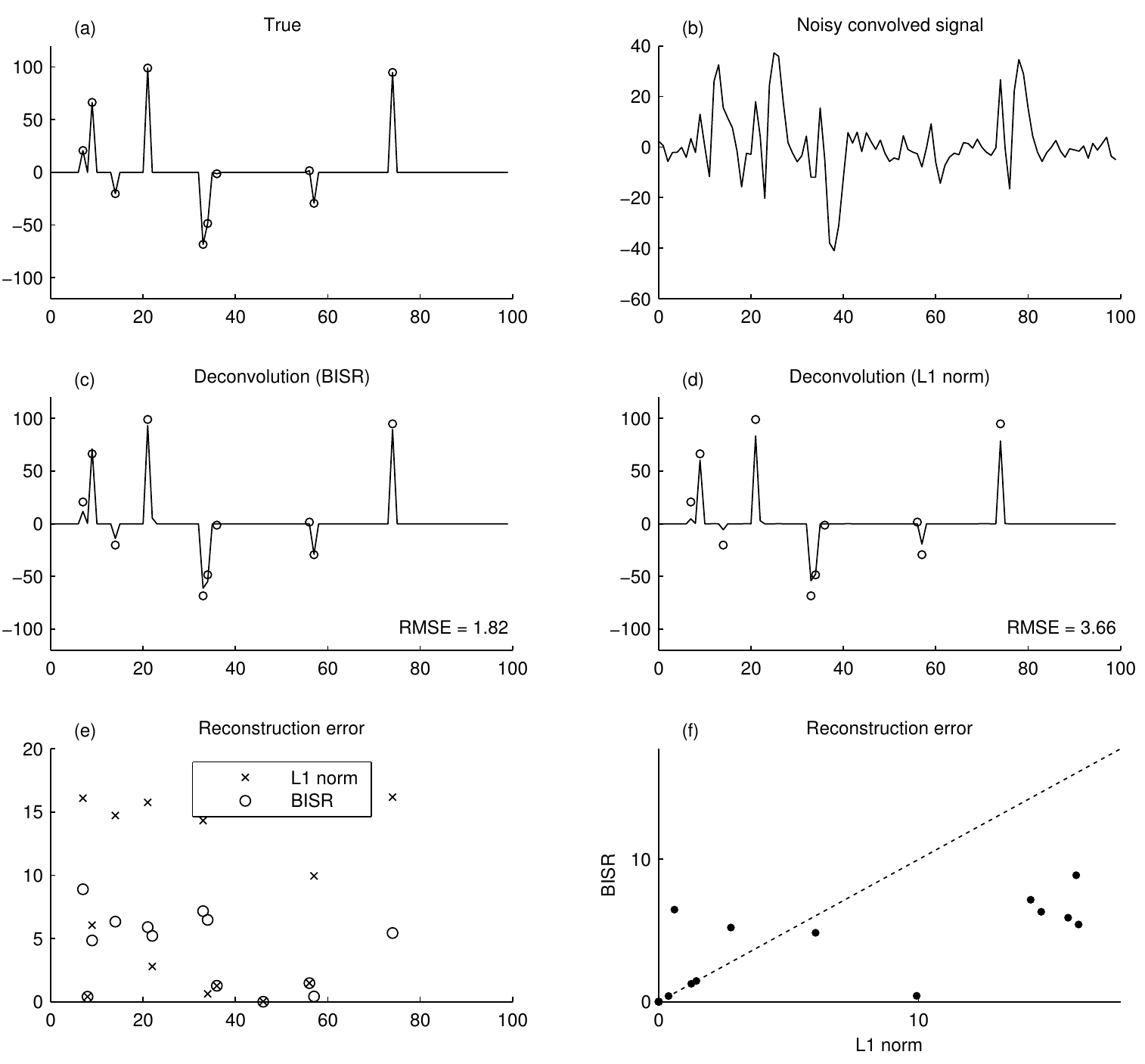}
	\caption{
		Example 2 of sparse deconvolution
		using bivariate sparse regularization (BISR).
	}
	\label{fig:deconv2}
\end{figure}

\begin{table}
	\caption{\label{table:rmse2}
		Average RMSE and run-time for Example 2}
	\begin{center}
	\begin{tabular}{@{} l ccccc r@{}} 
		\toprule 
Algorithm   & $ \sigma\!=\!1 $ & $ \sigma\!=\!2 $  & $ \sigma\!=\!4 $  & $ \sigma\!=\!8 $  & $ \sigma\!=\!16 $   & (msec)    \\ 
		\midrule 
L1              & 1.37       & 2.70       & 5.01       & 8.96       & 14.01  &   5.1 \\ 
L1+debiasing    & 0.76       & 1.55       & 3.14       & 6.56       & 13.31 &  5.2   \\ 
Lp (p = 0.5)    & 1.07       & 1.57       & 2.89       & 6.14       & 13.38    & 5.7  \\ 
SBR (L0)        & 0.73       & 1.44       & 2.73       & 6.59       & 17.64    & 3.4  \\ 
IMSC            & 0.54       & 1.18       & 2.69       & 6.26       & 12.38    & 305.9  \\ 
IPS             & 0.59       & 1.20       & 2.66       & 5.88       & 12.41   &    6.9 \\ 
BISR (log)      & 0.75       & 1.60       & 3.19       & 6.65       & 12.45  &    13.9 \\ 
BISR (rat)      & 0.73       & 1.56       & 3.08       & 6.49       & 12.37    &  14.6 \\ 
BISR (atan)     & 0.75       & 1.56       & 3.00       & 6.29       & 12.25   &   15.4 \\ 
	\bottomrule 
	\end{tabular} 
\end{center}
\end{table}

It is noteworthy in this example that 
we obtain a strictly non-convex penalty ensuring convexity of the objective function $ F $,
because in this example the convolution filter $ H $ is not invertible. 
It is possible only because the utilized penalty is non-separable. 

In Table \ref{table:rmse2} we compare the same algorithms as in Example 1. 
The regularization parameter $ \lam $ for each algorithm was chosen as in Example 1. 
The remarks made in Example 1 mostly apply here, but with a few exceptions as follows.
In contrast to the generally tendency, the $ \ell_p $ pseudo-norm
performed worse than the debiased $ \ell_1 $ norm solution for the lowest noise level ($ \sigma = 1 $).
It appears that the performance of the $ \ell_p $ pseudo-norm and SBR methods 
may degrade for low or high noise levels when
one attempts to set $ \lam $ proportional to $ \sigma $.
In addition, in contrast to Example 1, 
the BISR approach did not outperform 
$ \ell_p $ pseudo-norm or SBR for several noise levels. 
IPS performs very well at most noise levels.

\section{Limitations of Bivariate Penalties}
\label{sec:limits}

The examples in Sec.~\ref{sec:EGS} demonstrate the improvement
attainable by non-separable bivariate penalties in comparison with separable penalties;
however, the degree of improvement depends on the linear operator $ H $ in the 
data fidelity term.
In particular, for some $ H $, a bivariate penalty 
will offer little or no improvement in comparison with a separable penalty. 

In reference to Figs.~\ref{fig:filter1} and \ref{fig:filter2}, 
the effectiveness of a bivariate penalty for sparse deconvolution depends on how well
$ \abs{ H(\om) }^2 $ can be approximated by a function $ P(\om) $
of the form \eqref{eq:Pcos} satisfying condition \eqref{eq:ZPH}.
Some filters $ H $ can not be well approximated by such a filter $ P $.
For example, if $ H(\om) $ is Gaussian or has a sharp peak, 
then it can not be well approximated by such a filter $ P $.
In this case, a bivariate penalty
does not offer significant improvement in comparison with a separable penalty.
For a second example, 
if the frequency response $ H(\om) $ 
has multiple nulls (e.g., a $ K $-point moving average filter with $ K > 2 $),
then the only filter $ P $ of the form \eqref{eq:Pcos} 
satisfying condition \eqref{eq:ZPH} is identically zero, $ P(\om) = 0 $.
In this case, the bivariate penalty reduces to the \la\ norm and
the proposed BISR approach offers no improvement.

In both situations, 
a higher-order filter $ P $ is needed to accurately approximate $ H $,
and a higher-order non-separable penalty is needed to strongly induce sparsity.
Hence, to extend the applicability of non-separable sparse regularization, 
it will be necessary generalize the bivariate penalty to 
non-separable $ K $-variate penalties ($ K > 2 $).
Provided such an extension can be constructed, 
we expect the proposed BISR approach will be applicable to more general problems.

\section{Conclusion}

This paper aims to develop a convex approach for sparse signal estimation that improves upon  $ \ell_1 $ norm regularization, the standard convex approach.
We consider ill-conditioned linear inverse problems with a quadratic data fidelity term.
We focus in particular on the deconvolution problem. 
The proposed method is based on a non-convex penalty designed to ensure that the objective function is convex.
Our previous work \cite{Selesnick_2014_TSP_MSC} using this idea considered only separable (additive) penalties;
this is a fundamental limitation when the observation matrix is singular or near singular.

The non-separable bivariate penalty introduced in this paper overcomes this limitation.
The proposed bivariate sparse regularization (BISR) approach
provides a mechanism by which to improve upon $ \ell_1 $ norm regularization
while adhering to a convex framework.
The greater generality of non-separable regularization, as compared with separable regularization,
allows for the design of regularizers that more effectively induce sparsity.
Both BISR and \la\ norm regularization lead to convex optimization problems which can be solved
by similar optimization techniques. 

While we have focused on deconvolution,
we note that several filtering methods can be formulated as sparse deconvolution problems, 
e.g., peak detection \cite{Ning_2014_BEADS}
and denoising \cite{Selesnick_2015_SASS}.
Hence, the proposed approach (and extensions thereof) is not limited to deconvolution. 

\section*{Acknowledgment}

The authors gratefully acknowledge constructive comments from 
the anonymous reviewers.

\appendix



\section{Proofs}
\label{app:proofs}



\subsection{Proof of Proposition \ref{prop:gfconvex}}
\label{app:proofpen1}

\begin{proof}[Proof of Proposition \ref{prop:gfconvex}]
Let $ 0 \le a \le 1/\lam $.
By Proposition \ref{prop:sfun1d}, $ s $ is twice continuously differentiable.
Hence, so is $ g $. 
Thus, 
to show $ g $ is convex, we show its second derivative is non-negative. 
From \eqref{eq:sbound}, we have
\begin{equation}
	s''( t ; a ) \ge -a.
\end{equation}
Since $ a \le 1/\lam $, we have
\begin{equation}
	s''( t ; a ) \ge - \frac{ 1 }{ \lam }.
\end{equation}
Hence
\begin{equation}
	g''( t ) = 1 + \lam s''( t ; a ) \ge 0.
\end{equation}
Therefore, $ g $ is convex. 
From \eqref{eq:defs},
we have $ f( t ) = g( t ) + \lam \abs{  t  } $.
Since $ f $ is the sum of convex functions, $ f $ is convex.
\end{proof}

\subsection{Partial derivatives }

Some following proofs will require the partial derivatives of $ S(x; a) $.
From Definition \ref{def:Sfun}, they are as follows.
\begin{align}
	\label{eq:defS1}
	\frac{ \partial S }{ \partial x_1} ( x ; a ) & = 
	\begin{cases}
		s'( x_1 + r x_2 ; \alpha ) ,  & x \in \SP_1
		\\
		r s'( r x_1 + x_2 ; \alpha ) + (1 - r) \, s'( x_1 ; a_1 ) ,  & x \in \SP_2
		\\
		r s'( r x_1 + x_2 ; \alpha ) + (1 + r) \, s'( x_1 ; a_2 ) ,  & x \in \SP_3
		\\
		s'( x_1 + r x_2 ; \alpha )  ,  & x \in \SP_4
	\end{cases}
\end{align}

\begin{align}
	\label{eq:defS2}
	\frac{ \partial S }{ \partial x_2} ( x ; a ) & = 
	\begin{cases}
		r s'( x_1 + r x_2 ; \alpha ) + (1 - r) \, s'( x_2 ; a_1 ) ,  & x \in \SP_1
		\\
		s'( r x_1 + x_2 ; \alpha ) ,  & x \in \SP_2
		\\
		s'( r x_1 + x_2 ; \alpha ) ,  & x \in \SP_3
		\\
		r s'( x_1 + r x_2 ; \alpha ) + (1 + r) \, s'( x_2 ; a_2 ) ,  & x \in \SP_4
	\end{cases}
\end{align}

\begin{align}
	\label{eq:defS11}
	\frac{ \partial^2 S }{ \partial^2 x_1} ( x ; a ) & = 
	\begin{cases}
		s''( x_1 + r x_2 ; \alpha ) ,  & x \in \SP_1
		\\
		r^2 s''( r x_1 + x_2 ; \alpha ) + (1 - r) \, s''( x_1 ; a_1 ) ,  & x \in \SP_2
		\\
		r^2 s''( r x_1 + x_2 ; \alpha ) + (1 + r) \, s''( x_1 ; a_2 ) ,  & x \in \SP_3
		\\
		s''( x_1 + r x_2 ; \alpha )  ,  & x \in \SP_4
	\end{cases}
\end{align}

\begin{align}
	\label{eq:defS22}
	\frac{ \partial^2 S }{ \partial^2 x_2} ( x ; a ) & = 
	\begin{cases}
		r^2 s''( x_1 + r x_2 ; \alpha ) + (1 - r) \, s''( x_2 ; a_1 ) ,  & x \in \SP_1
		\\
		s''( r x_1 + x_2 ; \alpha ) ,  & x \in \SP_2
		\\
		s''( r x_1 + x_2 ; \alpha ) ,  & x \in \SP_3
		\\
		r^2 s''( x_1 + r x_2 ; \alpha ) + (1 + r) \, s''( x_2 ; a_2 ) ,  & x \in \SP_4
	\end{cases}
\end{align}

\begin{align}
	\label{eq:defS12}
	\frac{ \partial^2 S }{ \partial x_1 \partial x_2} ( x ; a ) & = 
	\begin{cases}
		r s''( x_1 + r x_2 ; \alpha ) ,  & x \in \SP_1
		\\
		r s''( r x_1 + x_2 ; \alpha )  ,  & x \in \SP_2
		\\
		r s''( r x_1 + x_2 ; \alpha )  ,  & x \in \SP_3
		\\
		r s''( x_1 + r x_2 ; \alpha )  ,  & x \in \SP_4
	\end{cases}
\end{align}

\subsection{Scaling identities }

Some following proofs will use scaling identities.
From the scaling property in Sec.~\ref{sec:pen1},
it follows that
\begin{align}
	s( t, \alpha ) 
	& = \frac{ a_1 }{ \alpha } \, s\Bigl( \frac{ \alpha }{ a_1 } t  ; a_1 \Bigr)
	\\
	\label{eq:stalpha}
	& = ( 1 + r ) \, s\Bigl( \frac{ t }{ 1 + r } ; a_1 \Bigr)
\end{align}
and similarly
\begin{align}
	s( t, a_1 ) 
	& = \frac{ 1 }{ 1 + r } \, s\bigl( ( 1 + r )  t ; \alpha \bigr)
	\\
	s'( t, a_1 ) 
	& =  s'\bigl( ( 1 + r )  t ; \alpha \bigr)
	\\
	\label{eq:s2da1}
	s''( t, a_1 ) 
	& =  (1 + r) s''\bigl( ( 1 + r )  t ; \alpha \bigr).
\end{align}
Likewise,
\begin{align}
	s( t, \alpha ) 
	& = \frac{ a_2 }{ \alpha } \, s\Bigl( \frac{ \alpha }{ a_2 } t ; a_2 \Bigr)
	\\
	\label{eq:sta2}
	& = ( 1 - r ) \, s\Bigl( \frac{ t }{ 1 - r } ; a_2 \Bigr)
\end{align}
and
\begin{align}
	s( t, a_2 ) 
	& = \frac{ 1 }{ 1 - r } \, s\bigl( ( 1 - r )  t ; \alpha \bigr)
	\\
	s'( t, a_2 ) 
	& =  s'\bigl( ( 1 - r )  t ; \alpha \bigr)
	\\
	\label{eq:s2da2}
	s''( t, a_2 ) 
	& =  (1 - r) s''\bigl( ( 1 - r )  t ; \alpha \bigr).
\end{align}


\subsection{Proof of  of Lemma \ref{lemma:defS}}
\label{sec:Sconcave}



\begin{proof}[Proof of Lemma \ref{lemma:defS}]


We first show that $ S $ in Definition \ref{def:Sfun} is consistent on the common boundaries of the sets $ A_i $.
Consider the common boundary of $ A_1 $ and $ A_4 $, i.e., the line $ \{ (x_1, 0) : x_1 \in \RR \} $.
From the $ A_1 $ side, 
\begin{align}
	S(x ; a )  \Big|_{\substack{x_2 = 0 \\ x \in A_1}}
	& = s( x_1 ; \alpha ) + (1 - r ) s( 0 ; a_1 )
	\\
	& = s( x_1 ; \alpha) 
\end{align}
where we used $ s( 0 ; \cdot ) = 0 $.
From the $ A_4 $ side,
\begin{align}
	S(x ; a )  \Big|_{\substack{x_2 = 0 \\ x \in A_4}}
	& = s( x_1 ; \alpha ) + (1 + r) \, s( 0 ; a_2 )
	\\
	& = s( x_1 ; \alpha).
\end{align}
Hence, $ S $ is consistently defined on the common boundary of $ A_1 $ and $ A_4 $.

Consider the common boundary of $ A_1 $ and $ A_2 $, i.e., the line $ \{ (x_1, x_1) : x_1 \in \RR \} $.
From the $ A_1 $ side, 
\begin{equation}
	S(x ; a )  \Big|_{\substack{x_2 = x_1 \\ x \in A_1}}
	= s( (1 + r) x_1 ; \alpha ) + (1 - r ) s( x_1 ; a_1 ).
\end{equation}
From the $ A_2 $ side,
\begin{equation}
	S(x ; a )  \Big|_{\substack{x_2 = x_1 \\ x \in A_2}}
	= s( (1+r) x_1 ; \alpha ) + (1 - r) \, s( x_1 ; a_2 ).
\end{equation}
Hence, $ S $ is consistently defined on the common boundary of $ A_1 $ and $ A_2 $.
Similarly, it can be shown that $ S $ is consistently defined on the other 
common boundaries of the sets $ A_i $.
Hence $ S $ is continuous because $ s $ is continuous.

We now show $ S $ is differentiable on $ \RR^2 $. 
We need only show $ S $ is differentiable on the common boundaries of the sets $ A_i $
because $ s $ is differentiable. 
Consider the common boundary of $ A_1 $ and $ A_4 $.
From the $ A_1 $ side,
\begin{equation}
	\frac{ \partial S }{ \partial x_1} ( x ; a ) 
	\Big|_{\substack{x_2 \to 0 \\ x \in A_1}}
	= s'( x_1 ; \alpha ) 
\end{equation}
using \eqref{eq:defS1},
and
\begin{align}
	\frac{ \partial S }{ \partial x_2} ( x ; a ) 
	\Big|_{\substack{x_2 \to 0 \\ x \in A_1}}
	& = r s'( x_1 ; \alpha ) + (1 - r) s'( 0 ; a_1 )
	\\
	& = r s'( x_1 ; \alpha )
\end{align}
using \eqref{eq:defS2}
and $ s'(0, \cdot) = 0 $.
From the $ A_4 $ side,
\begin{equation}
	\frac{ \partial S }{ \partial x_1} ( x ; a ) 
	\Big|_{\substack{x_2 \to 0 \\ x \in A_4}}
	= s'( x_1 ; \alpha ) 
\end{equation}
and
\begin{align}
	\frac{ \partial S }{ \partial x_2} ( x ; a ) 
	\Big|_{\substack{x_2 \to 0 \\ x \in A_4}}
	& = r s'( x_1 ; \alpha ) + (1 + r) s'( 0 ; a_2 )
	\\
	& = r s'( x_1 ; \alpha ).
\end{align}
Hence, the partial derivatives of $ S $ are continuous on the common of boundary of $ A_1 $ and $ A_4 $.

Consider the common boundary of $ A_1 $ and $ A_2 $.
From the $ A_1 $ side, 
\begin{align}
	\frac{ \partial S }{ \partial x_1} ( x ; a ) 
	\Big|_{\substack{x_2 \to x_1 \\ x \in A_1}}
	& = s'( (1 + r ) x_1 ; \alpha )
	\\
	& = s'( x_1 ; a_1 )
\end{align}
where we used the scaling property \eqref{eq:ssp1} and identity \eqref{eq:a1id}.
Also, using the same properties, 
\begin{align}
	\frac{ \partial S }{ \partial x_2} ( x ; a ) 
	\Big|_{\substack{x_2 \to x_1 \\ x \in A_1}}
	& = r s'( (1 + r ) x_1 ; \alpha ) + (1 - r) s'( x_1 ; a_1 )
	\\
	& = s'( x_1 ; a_1 ).
\end{align}
From the $ A_2 $ side, we similarly have
\begin{align}
	\frac{ \partial S }{ \partial x_1} ( x ; a ) 
	\Big|_{\substack{x_2 \to x_1 \\ x \in A_2}}
	& = r s'( (1 + r ) x_1 ; \alpha ) + (1 - r) s'( x_1 ; a_1 )
	\\
	& = s'( x_1 ; a_1 )
\end{align}
and
\begin{align}
	\frac{ \partial S }{ \partial x_2} ( x ; a ) 
	\Big|_{\substack{x_2 \to x_1 \\ x \in A_2}}
	& = s'( (1 + r) x_1 ; \alpha ).
	\\
	& = s'( x_1 ; a_1 ).
\end{align}
Hence, the partial derivatives of $ S $ are continuous on the common of boundary of $ A_1 $ and $ A_2 $.
Similarly, it can be shown that the partial derivatives of $ S $
are continuous on the other common boundaries of the sets $ A_i $.
Hence $ S $ is differentiable on $ \RR^ 2 $.

We now show $ S $ is twice differentiable. 
We need only show the second-order partial derivatives of $ S $
are continuous on the common boundaries of the sets $ A_i $
because $ s $ is twice differentiable. 


Consider the common boundary of $ A_1 $ and $ A_4 $.
Using \eqref{eq:defS11}, we obtain
\begin{equation}
	\frac{ \partial^2 S }{ \partial^2 x_1} ( x ; a ) 
	\Big|_{\substack{x_2 \to 0 \\ x \in A_1}}
	= s''( x_1 ; \alpha ) 
\end{equation}
and
\begin{equation}
	\frac{ \partial^2 S }{ \partial^2 x_1} ( x ; a ) 
	\Big|_{\substack{x_2 \to 0 \\ x \in A_4}}
	= s''( x_1 ; \alpha ).
\end{equation}
Hence $ \partial^2 S / \partial^2 x_1 $ is continuous on the common boundary of $ A_1 $ and $ A_4 $.
Using \eqref{eq:defS22}, we obtain
\begin{align}
	\frac{ \partial^2 S }{ \partial^2 x_2} ( x ; a ) 
	\Big|_{\substack{x_2 \to 0 \\ x \in A_1}}
	& = r^2 s''( x_1 ; \alpha ) + ( 1 - r ) s''(0 ; a_1)
	\\
	& = r^2 s''( x_1 ; \alpha ) - ( 1 - r ) a_1 
\end{align}
where we used $ s''(0; a_1) = -a_1 $,
and
\begin{align}
	\frac{ \partial^2 S }{ \partial^2 x_2} ( x ; a ) 
	\Big|_{\substack{x_2 \to 0 \\ x \in A_4}}
	& = r^2 s''( x_1 ; \alpha ) + ( 1 + r ) s''(0 ; a_2)
	\\
	& = r^2 s''( x_1 ; \alpha ) - ( 1 + r ) a_2 
\end{align}
where we used $ s''(0; a_2) = -a_2 $.
Using \eqref{eq:ralpha}, we have 
\begin{equation}
	(1 - r) a_1 = (1 + r) a_2 = \frac{ 2 a_1 a_2}{a_1 + a_2},
\end{equation}
hence $ \partial^2 S / \partial^2 x_2 $ is continuous on the common boundary of $ A_1 $ and $ A_4 $.

Consider the common boundary of $ A_1 $ and $ A_2 $.
Using \eqref{eq:defS11}, we obtain
\begin{equation}
	\frac{ \partial^2 S }{ \partial^2 x_1} ( x ; a ) 
	\Big|_{\substack{x_2 \to x_1 \\ x \in A_1}}
	= s''( (1 + r) x_1 ; \alpha ) 
\end{equation}
and
\begin{align}
	\frac{ \partial^2 S }{ \partial^2 x_1} ( x ; a ) 
	\Big|_{\substack{x_2 \to x_1 \\ x \in A_2}}
	& = r^2 s''( (1 + r) x_1 ; \alpha ) + (1 - r) s''(x_1 ; a_1)
	\\
	& = r^2 s''( (1 + r) x_1 ; \alpha ) + (1 - r) (1+r) s''( (1+r) x_1 ; \alpha)
	\\
	& = s''( (1 + r) x_1 ; \alpha ) 
\end{align}
where we used \eqref{eq:s2da1}.
Hence $ \partial^2 S / \partial^2 x_2 $ is continuous on the common boundary of $ A_1 $ and $ A_2 $.

Similarly, it can be shown that the remaining partial derivatives are continuous on
the other common boundaries of the sets $ A_i $.
Hence $ S $ is twice continuously differentiable on $ \RR^2 $.

We now show $ S $ is concave
by showing that $ \nabla^2 S(x) $ is negative semidefinite for all $ x \in \RR^2 $.
Specifically, 
we show that $ -[\nabla^2 S(x)] $ is positive semidefinite
by showing that its principal minors are non-negative (Sylvester's criterion).

The determinant of the Hessian of $ S $
\begin{equation}
	\det [\nabla^2 S(x)] =
		\frac{ \partial^2 S }{ \partial^2 x_1 }  
		\frac{ \partial^2 S }{ \partial^2 x_2 } 
		-
		\left( \frac{ \partial^2 S }{ \partial x_1 \partial x_2}  \right)^2
\end{equation}
is given by
\begin{align}
	&
	\det [\nabla^2 S(x; a)] =
	\\
	\nonumber
	&
	\qquad
	\begin{cases}
		(1 - r) \, s''( x_2 ; a_1 ) \, s''( x_1 + r x_2 ; \alpha ), & x \in \SP_1
		\\
		(1 - r) \, s''( x_1 ; a_1 ) \, s''( r x_1 + x_2 ; \alpha ), & x \in \SP_2
		\\
		(1 + r) \, s''( x_1 ; a_2 ) \, s''( r x_1 + x_2 ; \alpha ), & x \in \SP_3
		\\
		(1 + r) \, s''( x_2 ; a_2 ) \, s''( x_1 + r x_2 ; \alpha ), & x \in \SP_4.
	\end{cases}
\end{align} 
Since $ s''(t; \cdot ) $ is negative for all $ t $ and $ \abs{ r } $ is bounded by 1, 
it follows that $ 	\det [\nabla^2 S(x)] $ is non-negative for all $ x \in \RR^2 $.
Hence, the determinant of $ -[\nabla^2 S(x)] $ is non-negative.
Moreover,  $  \partial^2 S / \partial^2 x_i ( x ; a ) \le 0 $ for all $ x \in \RR^2 $
for $ i = 1, 2 $.
Hence, the principal minors of $ -[\nabla^2 S(x)] $ are non-negative.
This proves $ S $ is concave on $ \RR^2 $.
\end{proof}

%
%

\subsection{Proof of Lemma \ref{lemma:Shessian}}
\label{sec:Sdiff}

Using \eqref{eq:defS11}-\eqref{eq:defS12},  the Hessian of $ S $ is given by
\begin{equation}
	\nabla^2 S( x ; a )
	=
	\begin{cases}
	\begin{bmatrix}
		s''( x_1 + r x_2 ; \alpha )
		&
		r s''( x_1 + r x_2 ; \alpha )
		\\
		r s''( x_1 + r x_2 ; \alpha )
		&
		r^2 s''( x_1 + r x_2 ; \alpha ) + (1 - r) \, s''( x_2 ; a_1 )
	\end{bmatrix},
	\ \ 
	&
	x \in \SP_1
	\\[2em]
	\begin{bmatrix}
		r^2 s''( r x_1 + x_2 ; \alpha ) + (1 - r) \, s''( x_1 ; a_1 ) 
		&
		r s''( r x_1 + x_2 ; \alpha )
		\\
		r s''( r x_1 + x_2 ; \alpha )
		&
		s''( r x_1 + x_2 ; \alpha )
	\end{bmatrix},
	&
	x \in \SP_2
	\\[2em]
	\begin{bmatrix}
		r^2 s''( r x_1 + x_2 ; \alpha ) + (1 + r) \, s''( x_1 ; a_2 ) 
		&
		r s''( r x_1 + x_2 ; \alpha )  
		\\
		r s''( r x_1 + x_2 ; \alpha )  
		&
		s''( r x_1 + x_2 ; \alpha ) 
	\end{bmatrix},
	&
	x \in \SP_3
	\\[2em]
	\begin{bmatrix}
		s''( x_1 + r x_2 ; \alpha )  
		&
		r s''( x_1 + r x_2 ; \alpha )  
		\\
		r s''( x_1 + r x_2 ; \alpha )  
		&
		r^2 s''( x_1 + r x_2 ; \alpha ) + (1 + r) \, s''( x_2 ; a_2 ) 
	\end{bmatrix},
	&
	x \in \SP_4.
	\end{cases}
\end{equation}
We write this as
\begin{equation}
	\label{eq:ShessSA}
	\nabla^2 S( x ; a )
	=
	\begin{cases}
	s''(x_1 + r x_2; \alpha)
	\begin{bmatrix}
		1 & r
		\\
		r & r^2
	\end{bmatrix}
	+
	s''(x_2; a_1)
	\begin{bmatrix}
		0	&	0
		\\
		0	&	1 - r	
	\end{bmatrix},
	\ \ 
	&
	x \in \SP_1
	\\[2em]
	s''( r x_1 + x_2 ; \alpha )
	\begin{bmatrix}
		r^2 & r
		\\
		r & 1
	\end{bmatrix}
	+
	s''(x_1 ; a_1)
	\begin{bmatrix}
		1 - r &	0
		\\
		0	&	0
	\end{bmatrix},
	&
	x \in \SP_2
	\\[2em]
 	s''( r x_1 + x_2 ; \alpha ) 
	\begin{bmatrix}
		r^2 & r
		\\
		r & 1
	\end{bmatrix}
	+
	s''(x_1; a_2 )
	\begin{bmatrix}
		1 + r	&	0
		\\
		0	&	0
	\end{bmatrix},
	&
	x \in \SP_3
	\\[2em]
	s''( x_1 + r x_2 ; \alpha )  
	\begin{bmatrix}
		1 & r
		\\
		r & r^2
	\end{bmatrix}
	+
	s''(x_2 ; a_2 )
	\begin{bmatrix}
		0	&	0
		\\
		0	&	1 + r
	\end{bmatrix},
	&
	x \in \SP_4.
	\end{cases}
\end{equation}
%
Using \eqref{eq:s2da1} and \eqref{eq:s2da2}, we write \eqref{eq:ShessSA} as
\begin{equation}
	\label{eq:ShessS}
	\nabla^2 S( x ; a )
	=
	\begin{cases}
	s''(x_1 + r x_2; \alpha)
	\begin{bmatrix}
		1 & r
		\\
		r & r^2
	\end{bmatrix}
	+
	s''( (1 + r) x_2; \alpha)
	\begin{bmatrix}
		0	&	0
		\\
		0	&	1-r^2
	\end{bmatrix},
	\ \ 
	&
	x \in \SP_1
	\\[2em]
	s''( r x_1 + x_2 ; \alpha )
	\begin{bmatrix}
		r^2 & r
		\\
		r & 1
	\end{bmatrix}
	+
	s''( (1 + r) x_1; \alpha)
	\begin{bmatrix}
		1-r^2 &	0
		\\
		0	&	0
	\end{bmatrix},
	&
	x \in \SP_2
	\\[2em]
 	s''( r x_1 + x_2 ; \alpha ) 
	\begin{bmatrix}
		r^2 & r
		\\
		r & 1
	\end{bmatrix}
	+
	s''( (1 - r) x_1; \alpha)
	\begin{bmatrix}
		1-r^2 	&	0
		\\
		0	&	0
	\end{bmatrix},
	&
	x \in \SP_3
	\\[2em]
	s''( x_1 + r x_2 ; \alpha )  
	\begin{bmatrix}
		1 & r
		\\
		r & r^2
	\end{bmatrix}
	+
	s''( (1 - r) x_2; \alpha)
	\begin{bmatrix}
		0	&	0
		\\
		0	&	1-r^2
	\end{bmatrix},
	&
	x \in \SP_4.
	\end{cases}
\end{equation}

\begin{proof}[Proof of Theorem \ref{lemma:Shessian}]

Note that the matrices in \eqref{eq:ShessS} are positive semidefinite
because $ \abs{ r } \le 1 $.
We also have $ -\alpha \le s''(t; \alpha) < 0 $ for all $ t $.
Hence,
\begin{equation}
	\label{eq:nablaSgte}
	\nabla^2 S( x ; a )
	\mge
	-\alpha
	\begin{bmatrix}
		1	&	r
		\\
		r	&	1	
	\end{bmatrix},
	\ \ x \in \RR^2.
\end{equation}
%
%
Using \eqref{eq:ralpha} and \eqref{eq:adiffid}, 
we write \eqref{eq:nablaSgte} as
\begin{equation}
	\nabla^2 S( x ; a )
	\mge
	- \half
	\begin{bmatrix}
		a_1 + a_2 &  a_1 - a_2
		\\
		a_1 - a_2 & a_1 + a_2
	\end{bmatrix}.
\end{equation}
Since $  s''(0; \alpha) = -\alpha $   [see \eqref{eq:s2a}],
we similarly have
\begin{equation}
	\nabla^2 S( 0 ; a )
	=
	- \half
	\begin{bmatrix}
		a_1 + a_2 &  a_1 - a_2
		\\
		a_1 - a_2 & a_1 + a_2
	\end{bmatrix}.
\end{equation}

\end{proof}

\subsection{Proof of Theorem \ref{thm:bounds}}
\label{sec:bounds}

\begin{proof}[Proof of Theorem \ref{thm:bounds}]

Without loss of generality, assume $ a_1 > a_2 $.
From	 \eqref{eq:defpsi} and \eqref{eq:defs}, 
the inequality \eqref{eq:bounds} can be written as
\begin{equation}
	\label{eq:Sbounds}
	s(x_1; a_1) + s(x_2; a_1)
	\le S(x; a) \le 
	s(x_1; a_2) + s(x_2; a_2)
\end{equation}
where $ S $ is given in Definition \ref{def:Sfun}.
First we prove \eqref{eq:Sbounds} for $ x \in \SP_1 $.
Second, we prove it for $ x \in \SP_3 $.
The proofs for $ x \in \SP_2 $ and $ x \in \SP_4 $ are essentially identical by symmetries. 
Note that since $ a_1 \ge a_2 \ge 0 $, we have 
$ a_1 \ge \alpha \ge a_2 $
and
$ 0 \le r \le 1 $.

Let $ x \in \SP_1 $. We seek to prove
\begin{equation}
	\label{eq:A1L}
	s( x_1 + r x_2 ; \alpha ) + (1 - r) \, s( x_2 ; a_1 )
	\ge 
	s(x_1; a_1) + s(x_2; a_1).
\end{equation}
Using \eqref{eq:stalpha}, we have
\begin{align}
	s( x_1 + r x_2 ; \alpha )
	& =
	(1 + r ) \, s\Bigl( \frac{ x_1 + r x_2 }{ 1 + r } ; a_1 \Bigr)
	\\
	& \ge 
	( 1 + r ) \, \Bigl[  \frac{ 1 }{ 1 + r } s( x_1 ; a_1 ) + \frac{ r }{ 1 + r } s( x_2 ; a_1 ) \Bigr]
	\\
	\label{eq:sx1rx2}
	& = 
	s( x_1 ; a_1 ) + r \, s( x_2 ; a_1 )
\end{align}
where the inequality is due to $ s $ being a concave function.
Adding $ (1 - r) \, s( x_2 ; a_1 ) $ to both sides of \eqref{eq:sx1rx2} gives \eqref{eq:A1L}.

Let $ x \in \SP_1 $. 
(That is, $ 0 \le x_2 \le x_1 $
or
$ x_1 \le x_2 \le 0 $.)
We seek to prove
\begin{equation}
	\label{eq:A1U}
	s(x_1; a_2) + s(x_2; a_2)
	\ge
	s( x_1 + r x_2 ; \alpha ) + (1 - r) \, s( x_2 ; a_1 ).
\end{equation}
From \eqref{eq:stalpha} we have
\begin{equation}
	\label{eq:sx1a1}
	s( x_2 ; a_1 ) = \frac{ 1 }{ 1 + r }  \, s( (1+r) x_2 ; \alpha ).
\end{equation}
Suppose $ 0 \le x_2 \le x_1 $.
As $ s $ is concave and $ s(0 ; \cdot \,) = 0 $, we have
\begin{equation}
	\label{eq:sx1alpha}
	s( x_1 ; \alpha ) \ge \frac{ x_1 }{ x_1 + r x_2 } \, s( x_1 + r x_2 ; \alpha )
\end{equation}
because $ 0 \le x_1 \le x_1 + r x_2 $.
Similarly,
\begin{equation}
	\label{eq:s1rx2}
	s( (1+r) x_2 ; \alpha ) \ge 
	\frac{ (1 + r) x_2 }{ x_1 + r x_2 } \, s( x_1 + r x_2 ; \alpha )
\end{equation}
because $ 0 \le (1 + r ) x_2 \le x_1 + r x_2 $.
From \eqref{eq:sx1a1} and \eqref{eq:s1rx2}
it follows that
\begin{equation}
	\label{eq:rsx2a1}
	r s( x_2 ; a_1 ) \ge 
	\frac{ r x_2 }{ x_1 + r x_2 } \, s( x_1 + r x_2 ; \alpha ).
\end{equation}
Adding \eqref{eq:sx1alpha} and \eqref{eq:rsx2a1} gives
\begin{equation}
	\label{eq:sx1rsx2sx1}
	s( x_1 ; \alpha ) + r s( x_2 ; a_1 ) \ge s( x_1 + r x_2 ; \alpha ).
\end{equation}
Adding $ (1 - r ) \, s( x_2 ; a_1 ) $ to both sides of \eqref{eq:sx1rsx2sx1} gives
\begin{equation}
	s( x_1 ; \alpha ) + s( x_2 ; a_1 ) \ge s( x_1 + r x_2 ; \alpha ) + (1 - r ) \, s( x_2 ; a_1 ).
\end{equation}
Since $ \alpha \ge a_2 $, we have $ s( x_1 ; a_2 ) \ge s( x_1 ; \alpha )  $.
Similarly, 
since  $ a_1 \ge a_2 $, we have $ s( x_2 ; a_2 ) \ge s( x_2 ; a_1 )  $.
It follows that \eqref{eq:A1U} holds.
For $ x_1 \le x_2 \le 0 $ the proof is similar.

Let $ x \in \SP_3 $.
We seek to prove
\begin{equation}
	\label{eq:A3L}
	s( r x_1 + x_2 ; \alpha ) + (1 + r) \, s( x_1 ; a_2 )
	\ge 
	s(x_1; a_1) + s(x_2; a_1).
\end{equation}
Suppose $ x_2 \ge -x_1 \ge 0  $.
Since $ 0 \le r \le 1 $ it follows that $ 0 \le r x_1 + x_2 \le x_ 2 $.
As $ s $ is concave and $ s(0 ; \cdot \,) = 0 $, we have
\begin{equation}
	\label{eq:sx1aB}
	s( r x_1 + x_2 ; \alpha ) \ge \frac{ r x_1 + x_2 }{ x_2 } \, s( x_2 ; \alpha )
\end{equation}
and
\begin{equation}
	\label{eq:sx1mx1}
	s( x_1 ; \alpha ) = s( -x_1 ; \alpha) \ge -\frac{ x_1 }{ x_2 } s( x_2 ; \alpha )
\end{equation}
where we use the fact that $ s $ is symmetric. 
Multiplying \eqref{eq:sx1mx1} by $ r $ and adding \eqref{eq:sx1aB}, we have
\begin{equation}
	\label{eq:srxg}
	s( r x_1 + x_2 ; \alpha ) + r s( x_1 ; \alpha ) \ge s( x_2 ; \alpha )
\end{equation}
Since $ a_2 \le \alpha $, we have
$ s( x_1 ; a_2 ) \ge s( x_1 ; \alpha) $.
Multiplying by $ (1 + r) $ and rearranging, we have
\begin{equation}
	\label{eq:srxf}
	(1 + r ) s( x_1 ; a_2 ) - r s( x_1 ; \alpha ) \ge s( x_1 ; \alpha ).
\end{equation}
Adding \eqref{eq:srxg} and \eqref{eq:srxf}, we have
\begin{equation}
	s( r x_1 + x_2 ; \alpha ) + (1 + r) \, s( x_1 ; a_2 )
	\ge 
	s(x_1; \alpha ) + s(x_2; \alpha ).
\end{equation}
Since $ \alpha \le a_1 $, we have $ s( t ; \alpha) \ge s(t ; a_1) $ for all $ t $.
It follows that \eqref{eq:A3L} holds.
For $ x_2 \le -x_1 \le 0 $ the proof is similar.

Let $ x \in \SP_3 $.
We seek to prove
\begin{equation}
	\label{eq:A3U}
	s( r x_1 + x_2 ; \alpha ) + (1 + r) \, s( x_1 ; a_2 )
	\le 
	s(x_1; a_2) + s(x_2; a_2).
\end{equation}
From \eqref{eq:sta2} we have
\begin{equation}
	s( r x_1 + x_2 ; \alpha ) = ( 1 - r )  \, s\Bigl( \frac{ r x_1 + x_2 }{ 1 - r } ; a_2 \Bigr).
\end{equation}
Since $ s $ is symmetric, we have $ s( x_1 ; a_2 ) = s( -x_1 ; a_2 ) $.
Hence
\begin{align}
	s( r x_1 + x_2 ; \alpha ) + r s( x_1 ; a_2 ) 
	& = 
	( 1 - r )  \, s\Bigl( \frac{ r x_1 + x_2 }{ 1 - r } ; a_2 \Bigr)
	+ r s( -x_1 ; a_2 )
	\\
	& \le
	s\Bigl( ( 1 - r ) \frac{ r x_1 + x_2 }{ 1 - r } - r x_1 ; a_2 \Bigr)
	\\
	\label{eq:sx2a2p}
	& =
	s( x_2 ; a_2 )
\end{align}
where the inequality is due to $ s $ being a concave function.
Adding $ s( x_1 ; a_2 ) $ to both sides of \eqref{eq:sx2a2p} gives the
desired inequality \eqref{eq:A3U}.
\end{proof}

\subsection{Proof of Lemma \ref{lemma:tridiag}}
\label{sec:lemma:tridiag}

\begin{proof}[Proof of Lemma \ref{lemma:tridiag}]
We express $ F $ in \eqref{eq:defFdeconv} as
\begin{align}
	F(x) & = \half \norm{ y - H x }_2^2 + \frac{\lam}{2} \sum_n \psi( (x_{n-1}, x_n) ; a) 
	\\
	& =
	\half y\tr y - y\tr H x + \half x\tr \, ( H\tr H - P) \, x 
		+
		g( x )
\end{align}
where $ g \colon \RR^N \to \RR $ is defined as
\begin{equation}
	g(x) = 
		\half x\tr P x + \frac{\lam}{2} \sum_n \psi( (x_{n-1}, x_n) ; a).
\end{equation}
The first three terms are convex. 
Hence $ F $ is convex if $ g $ is  convex.
We express $ P $ as a sum of matrices, each comprising a single $ 2 \times 2 $ block.
\begin{equation}
	P = 
	\cdots
	+
	\begin{bmatrix}
		0 &  & & & \\
		& \half p_0 & p_1  & & \\
		& p_1 & \half p_0 &  & \\
		&  & & 0  & \\
		&  &  &  & 0
	\end{bmatrix}	
	+
	\begin{bmatrix}
		0 & & & & \\
		& 0 & & & \\
		& & \half p_0 & p_1  & \\
		&  & p_1 & \half p_0 & \\
		&  &  & & 0
	\end{bmatrix}	
	+
	\ \cdots \
\end{equation}
Hence $ g $ may be expressed as
\begingroup
\renewcommand{\arraystretch}{0.8}%
\begin{align}
	g(x) & = 
	\half \sum_n 
	\left(
	\big( x_{n-1}, \, x_n \big)
	\begin{bmatrix}
		\half p_0 & p_1 \\
		p_1 & \half p_0
	\end{bmatrix}
	\begin{pmatrix} x_{n-1} \\ x_n \end{pmatrix}
	+
	\lam \psi( (x_{n-1}, x_n) ; a)
	\right)
	\\
	& =
	\half \sum_n f( (x_{n-1}, x_n) )
\end{align}
\endgroup
where $ f \colon \RR^2 \to \RR $ is given by \eqref{eq:defp}.
If $ f $ is convex, then $ g $ is the sum of convex functions and is thus convex itself. 
\end{proof}

\bibliographystyle{plain}

\end{document}